\documentclass[10pt]{article}
\usepackage{amssymb}
\usepackage{amssymb}
\usepackage{amsmath}
\usepackage{amsfonts}
\usepackage{graphicx}
\usepackage{soul, color}
\usepackage{tikz}
\usepackage{hyperref}
\usepackage{multirow}

\newtheorem{theorem}{Theorem}

\newtheorem{corollary}[theorem]{Corollary}

\newtheorem{definition}[theorem]{Definition}

\newtheorem{lemma}[theorem]{Lemma}

\newtheorem{proposition}[theorem]{Proposition}
\newtheorem{remark}[theorem]{Remark}

\newenvironment{proof}[1][Proof]{\noindent\textbf{#1.} }{\ \rule{0.5em}{0.5em}}
\input{tcilatex}

\newcommand{\ud}{\,\mathrm{d}}
\newcommand{\p}{\ensuremath{\partial}}
\newcommand{\n}{\ensuremath{\nonumber}}
\newcommand{\eps}{\ensuremath{\varepsilon}}
\newcommand{\bigO}{\mathcal{O}}

\DeclareSymbolFont{rmlargesymbols}{OMX}{mdbch}{m}{n}
\DeclareMathSymbol{\rmintop}{\mathop}{rmlargesymbols}{82}

 \allowdisplaybreaks

\title{\vspace{-50pt} Steady Prandtl Layer Expansions with External Forcing}
\author{ \Large Yan Guo\footnote{\url{yan_guo@brown.edu}. Division of Applied Mathematics, Brown University, 182 George Street, Providence, RI 02912, USA.} \hspace{10 mm} Sameer Iyer \footnote{\url{ssiyer@math.princeton.edu}. Department of Mathematics, Princeton University, Fine Hall, Washington Road, Princeton, NJ 08540, USA.}}
\date{September 6, 2018}

\begin{document}

\maketitle

\begin{abstract}
In this article we apply the machinery developed in \cite{GI1} together with a new compactness estimate and an object called the degree in order to prove validity of steady Prandtl layer expansions with external forcing.  
\end{abstract}

\section{Introduction and Notation}

We consider the steady, incompressible Navier-Stokes equations on the two-dimensional domain, $(x,Y) \in \Omega = (0,L) \times (0, \infty)$. Denoting the velocity $\bold{U}^{NS} := (U^{NS}, V^{NS})$, the equations read: 
\begin{align}  \label{intro.NS}
\left.
\begin{aligned}
&\bold{U}^{NS} \cdot \nabla \bold{U}^{NS} + \nabla P^{NS} = \eps \Delta \bold{U}^{NS} + \bold{g}_{ext} \\
&\nabla \cdot \bold{U}^{NS} = 0 
\end{aligned}
\right\} \text{ in } \Omega
\end{align}

The system above is taken with the no-slip boundary condition on $\{Y = 0\}$: 
\begin{align} \label{noslip.BC}
[U^{NS}, V^{NS}]|_{Y= 0} = [0, 0]. 
\end{align}

Here, the function $\bold{g}_{ext} = (g^{(u)}_{ext}, g^{(v)}_{ext})$ is an external force which vanishes in the inviscid limit. The form of the forcing we treat is given in (\ref{form.force}). 

In this article, we fix an outer Euler shear flow of the form $[u^0_e(Y), 0, 0]$, (satisfying certain assumptions given in (\ref{charlie.1}) - (\ref{charlie.3})). A fundamental question is to describe the asymptotic behavior of solutions to (\ref{intro.NS}) as the viscosity vanishes, that is as $\eps \rightarrow 0$. Generically, there is a mismatch of the tangential velocity at the boundary $\{Y = 0\}$ of the viscous flows, (\ref{noslip.BC}), and inviscid flows. Thus, one cannot expect $[U^{NS}, V^{NS}] \rightarrow [u^0_e, 0]$ in a sufficiently strong norm (for instance, $L^\infty$).  

To rectify this mismatch, it was proposed in 1904 by Ludwig Prandtl that there exists a thin fluid layer of size $\sqrt{\eps}$ near the boundary $Y = 0$ that bridges the velocity of $U^{NS}|_{Y = 0} = 0$ with the nonzero Eulerian velocity. This layer is known as the Prandtl boundary layer. 

We work with the scaled boundary layer variable:
\begin{align} \label{BL.variable}
y = \frac{Y}{\sqrt{\eps}},
\end{align}

Consider the scaled Navier-Stokes velocities: 
\begin{align}
U^\eps(x,y) = U^{NS}(x,Y), \hspace{3 mm} V^\eps = \frac{V^{NS}(x,Y)}{\sqrt{\eps}}, \hspace{3 mm} P^\eps(x,y) = P^{NS}(x,Y). 
\end{align}

Equation (\ref{intro.NS}) now becomes: 
\begin{align} \label{scaledNSint}
\begin{aligned}
&U^\eps  U^\eps_x + V^\eps U^\eps_y + P^\eps_x = \Delta_\eps U^\eps +  g^{(u)}_{ext} \\
&U^\eps V^\eps_x + V^\eps V^\eps_y + \frac{P^\eps_y}{\eps} = \Delta_\eps V^\eps + \frac{g^{(v)}_{ext}}{\sqrt{\eps}}\\
&U^\eps_x + V^\eps_y = 0
\end{aligned}
\end{align}

We expand the solution in $\eps$ as:
\begin{align}
\begin{aligned} \label{exp.u}
&U^\eps = u^0_e + u^0_p + \sum_{i = 1}^n \sqrt{\eps}^i (u^i_e + u^i_p) + \eps^{p} u^{(\eps)} := u_s + \eps^{N_0} u^{(\eps)}, \\
&V^\eps = v^0_p + v^1_e + \sum_{i = 1}^{n-1} \sqrt{\eps}^i (v^i_p + v^{i+1}_e) + \sqrt{\eps}^n v^n_p + \eps^{N_0} v^{(\eps)} := v_s + \eps^{N_0} v^{(\eps)}, \\
&P^\eps = P^0_e + P^0_p + \sum_{i = 1}^n \sqrt{\eps}^i (P^i_e + P^i_p) + \eps^{p} P^{(\eps)} := P_s + \eps^{N_0} P^{(\eps)},
\end{aligned}
\end{align}

\noindent where the coefficients are independent of $\eps$. Here $[u^i_e, v^i_e]$ are Euler correctors, and $[u^i_p, v^i_p]$ are Prandtl correctors. These are constructed in the Appendix, culminating in Theorem \ref{thm.construct}.

Correspondingly, we expand the forcing into $\bold{g}_{ext} :=  \begin{pmatrix} g^{(u)}_{ext} \\ g^{(v)}_{ext}\end{pmatrix}$, which is given by \footnote{Our result also applies if we add a leading order term $g^{v,0}_{ext, e} = -1$, which accounts for gravity.} 
\begin{align} \label{form.force}
\begin{pmatrix} \sum_{i = 1}^n \sqrt{\eps}^i (g^{u,i}_{ext, e} + g^{u,i}_{ext,p}) + \eps^{N_0} g^{u,(\eps)}_{ext} \\ \sqrt{\eps} [g^{v,1}_{ext, e} + \sum_{i = 1}^{n-1} \sqrt{\eps}^i (g^{v,i}_{ext, p} + g^{v,i+1}_{ext, e}) + \sqrt{\eps}^n g^{v,n}_{ext, p} + \eps^{N_0} g^{v,\eps}_{ext}]  \end{pmatrix}
\end{align}

\noindent The main assumption on the forcing applies to $g^{u,1}_{ext, p}$, and is given in (\ref{charlie.3}).

For our analysis, we will take $n = 4$ and $p = \frac{3}{2}+$. Let us also introduce the following notation: 
\begin{align}
\begin{aligned} \label{intro.bar.prof}
 \bar{u}^i_p := u^i_p - u^i_p|_{y = 0}, \hspace{3 mm} \bar{v}^i_p := v^i_p - v^1_p|_{y = 0}, \hspace{3 mm} \bar{v}^i_e := v^i_e - v^i_e|_{Y = 0}.
\end{aligned}
\end{align}

The profile $\bar{u}^0_p, \bar{v}^0_p$ from (\ref{intro.bar.prof}) is classically known as the ``boundary layer"; one sees from (\ref{exp.u}) that it is the leading order approximation to the Navier-Stokes velocity, $U^\eps$. The final layer, 
\begin{align*}
[u^{(\eps)}, v^{(\eps)}, P^{(\eps)}] = [\bold{u}^\eps, P^\eps].
\end{align*}

\noindent are called the ``remainders" and importantly, they depend on $\eps$. Controlling the remainders uniformly in $\eps$ is the fundamental challenge in order to establish the validity of (\ref{exp.u}), and the centerpiece of our article. 

We begin by briefly discussing the approximations, $[u_s, v_s]$. The particular equations satisfied by each term in $[u_s, v_s]$ is recorded in Appendix \ref{appendix.derive}. We record Theorem \ref{thm.construct}, which is proven in companion paper \cite{GI1}. We are prescribed the shear Euler flow, $u^0_e$. The profiles $[u^i_p, v^i_p]$ are Prandtl boundary layers. Importantly, these layers are rapidly decaying functions of the boundary layer variable, $y$. At the leading order, $[u^0_p, v^0_p]$ solve the nonlinear Prandtl equation: 
\begin{align}
\begin{aligned} \label{Pr.leading.intro}
&\bar{u}^0_p u^0_{px} + \bar{v}^0_p u^0_{py} - u^0_{pyy} + P^0_{px} = 0, \\
&u^0_{px} + v^0_{py} = 0, \hspace{3 mm} P^0_{py} = 0, \hspace{3 mm} u^0_p|_{x = 0} = U^0_P, \hspace{3 mm} u^0_p|_{y = 0} = - u^0_e|_{Y = 0}.
\end{aligned}
\end{align}

Soon after Prandtl's seminal 1904 paper, Blasius discovered the celebrated self-similar solution to (\ref{Pr.leading.intro}) (with zero pressure). This solution reads 
\begin{align} \label{blasius}
[\bar{u}^0_p, \bar{v}^0_p] = \Big[f'(\eta), \frac{1}{\sqrt{x + x_0}}\{ \eta f'(\eta) - f(\eta) \} \Big], \text{ where } \eta = \frac{y}{\sqrt{x+x_0}},
\end{align}

\noindent where $f$ satisfies
\begin{align} \label{blasius.ODE}
ff'' + f''' = 0, \hspace{3 mm} f'(0) = 0, \hspace{2 mm} f'(\infty) = 1, \hspace{2 mm} \frac{f(\eta)}{\eta} \xrightarrow{n \rightarrow \infty} 1.
\end{align}

\noindent Here, $x_0 \ge 0$ is a free parameter. It is well known that $f''(\eta)$ has a Gaussian tail, and that the following hold: 
\begin{align*}
0 \le f' \le 1, \hspace{3 mm} f''(\eta) \ge 0, \hspace{3 mm} f''(0) > 0, \hspace{3 mm} f'''(\eta) < 0 . 
\end{align*}

\noindent Moreover, the normal velocity satisfies the asymptotics:
\begin{align} \label{sign.blasius}
\lim_{\eta \rightarrow \infty} \frac{1}{\sqrt{x+x_0}} \{ \eta f'(\eta) - f(\eta) \} = v^1_e|_{Y = 0}  > 0  
\end{align}

Such a Blasius profile has been confirmed by experiments with remarkable accuracy as the main validation of the Prandtl theory (see \cite{Schlicting} for instance). These profiles are also canonical from a mathematical standpoint in the following sense: the work, \cite{Serrin}, has proven that when $x$ gets large (downstream), solutions to the Prandtl equation, (\ref{Pr.leading.intro}), converge to an appropriately renormalized Blasius profile. Therefore, validating the expansions (\ref{exp.u}) for the Blasius profile is the main objective and motivation in our study. 

Our main assumptions are imposed on the prescribed $\{x = 0\}$ data for the first order Euler approximation for the normal velocity and on the leading order shear flow. Denote by $u_{\parallel} = \bar{u}^0_p|_{x = 0}$, and $v_{\parallel} = \bar{v}^0_p|_{x = 0}$. To state our assumptions we need to define the notion of ``degree":
\begin{definition} \normalfont The degree, $d$, of a function $f(y)$ is
\begin{align} \label{defn.boldd}
\bold{d}(f) := \int_0^\infty K(y) I_y[f] \ud y, \hspace{3 mm} K(y) := u_{\parallel} e^{-\int_1^y v_{\parallel}}, \hspace{3 mm} I_y[f] := - \int_y^\infty f(z) \ud z. 
\end{align}

\end{definition}

We now delineate our main assumptions: 
\begin{subequations}
\begin{align}  \label{charlie.1}
&v^1_e|_{x = 0} > 0 \text{ for } Y \ge 0,  \\ \label{charlie.2}
&\| u^0_{eYY} \langle Y \rangle^2 \|_\infty = o(1), \\ \label{charlie.3}
&|\int_0^\infty K(y) g^{u,1}_{ext, p}(y) \ud y| \gtrsim 1.   
\end{align}
\end{subequations}

\noindent We will refer to (\ref{charlie.3}) as the ``non-degeneracy condition". As we will point out below, these assumptions are certainly consistent with the Blasius layer, (\ref{blasius}), and include a large class of shear flows for $u^0_e$ (see Subsection \ref{section.euler.example}).

We will also require the following finiteness assumptions on the Euler layers:
\begin{align}
\begin{aligned} \label{assume.Euler.i}
&v^1_e|_{x = 0} \text{ decays either algebraically or exponentially as } Y \uparrow \infty, \\
&\sup_Y |\frac{v^i_e|_{x = 0}}{v^1_e|_{x = 0}}| < \infty.
\end{aligned}
\end{align}

The system satisfied by the remainders, $[u^{(\eps)},v^{(\eps)},P^{(\eps)}]$, in vorticity formulation gives: 
\begin{align} 
\begin{aligned} \label{eqn.vort.intro}
-R[q^{(\eps)}] - u^{(\eps)}_{yyy} &+ 2\eps v^{(\eps)}_{xyy} + \eps^2 v^{(\eps)}_{xxx} + v_s \Delta_\eps u^{(\eps)} - u^{(\eps)} \Delta_\eps v_s \\
& = \eps^{N_0} \{u^{(\eps)} \Delta_\eps v^{(\eps)} - v^{(\eps)} \Delta_\eps v^{(\eps)}\} + g,
\end{aligned}
\end{align}

\noindent Here, $\Delta_\eps := \p_{yy} + \eps \p_{xx}$, $g$ is a forcing term that we will not discuss further in the introduction, and where we have defined the Rayleigh operator:
\begin{align} \label{rayleigh.quotient}
R[q^{(\eps)}] = \p_y\{ u_s^2 \p_y q^{(\eps)}\} + \eps \p_x \{ u_s^2 q^{(\eps)}_x \}, \hspace{3 mm} q^{(\eps)} := \frac{v^{(\eps)}}{u_s}.
\end{align}

\noindent The boundary condition we take are the following: 
\begin{align}
\begin{aligned} \label{rene}
&v^{(\eps)}|_{x = 0} = a^\eps_0(y), v^{(\eps)}_x|_{x = L} = a^\eps_1(y), v^{(\eps)}_{xx}|_{x = 0} = a^\eps_2(y), v^{(\eps)}_{xxx} = a^\eps_3(y) \\
&v^{(\eps)}|_{y = 0} = v^{(\eps)}_y|_{y = 0} = u^{(\eps)}|_{y = 0} = 0, v^{(\eps)}|_{y \uparrow \infty} = 0, 
\end{aligned}
\end{align}

\noindent Here, the $a^\eps_i(y)$ are prescribed boundary data which we assume satisfy: 
\begin{align} \label{assume.bq.intro}
\| \p_y^{j} a^\eps_i \{ \frac{1}{\sqrt{\eps}} \langle y \rangle^{\frac{1}{2}+} + \frac{\langle y \rangle}{v^1_e|_{x = 0}} \} \| \le o(1) \text{ for } j = 0,...,4,
\end{align}

\noindent which is a quantitative statement that the expansion (\ref{exp.u}) is valid at $\{x = 0\}$ and $\{x = L\}$. 

We are now able to state our main result, so long as we remain vague regarding the norm $\| \cdot \|_{\mathcal{X}}$ that appears below. A discussion of this norm will be in Subsection \ref{subsection.Main}.

\begin{theorem}[Main Theorem] \label{thm.main}  Assume boundary data and forcing are prescribed as in  Theorem \ref{thm.construct} and satisfying the assumptions (\ref{charlie.1}) - (\ref{charlie.3}), (\ref{assume.Euler.i}), and (\ref{assume.bq.intro}). Let $0 < \eps << L << 1$. Then all terms in the expansion (\ref{exp.u}) exist and are regular, $\| u_s, v_s \|_\infty \lesssim 1$. The remainders, $[u^{(\eps)},v^{(\eps})]$ exists uniquely in the space $\mathcal{X}$ and satisfy: 
\begin{align} \label{unif.x}
\| \bold{u}^\eps \|_{\mathcal{X}} \lesssim 1.
\end{align}
\noindent The Navier-Stokes solutions satisfy:
\begin{align} \label{main.inviscid}  
\| U^{NS} - u^0_e - u^0_p \|_\infty \lesssim \sqrt{\eps} \text{ and } \| V^{NS} - \sqrt{\eps} v^0_p - \sqrt{\eps}v^1_e \|_\infty \lesssim \eps.
\end{align}
\end{theorem}

Upon establishing the uniform bound (\ref{unif.x}), the result (\ref{main.inviscid}) follows from the following inequalities: $\| v \|_\infty \lesssim  \eps^{-\frac{1}{2}} \| \bold{u} \|_{\mathcal{X}}$, and $\|u \|_{\infty} \lesssim  \eps^{-\frac{(1+)}{2}} \|\bold{u} \|_{\mathcal{X}}$. These are established in Lemmas \ref{lemma.unif.u} and \ref{lemma.nonlinear} together with the definitions in (\ref{defn.norms.ult}).
 
We also note that thanks to (\ref{sign.blasius}), the assumption (\ref{charlie.1}) is consistent with the Blasius profile.

\subsection{Notation}

Before we state the main ideas of the proof, we will discuss our notation. Since we use the $L^2$ norm extensively in the analysis, we use $\| \cdot \|$ to denote the $L^2$ norm. It will be clear from context whether we mean $L^2(\mathbb{R}_+)$ or $L^2(\Omega)$. When there is a potential confusion (for example, when changing coordinates), we will take care to specify with respect to which variable the $L^2$ norm is being taken (for instance, $L^2_y$ means with respect to $\ud y$, whereas $L^2_Y$ will mean with respect to $\ud Y$). Similarly, when there is potential confusion, we will distinguish $L^2$ norms along a one-dimensional surface (say $\{x = 0\}$) by $\| \cdot \|_{x = 0}$. Analogously, we will often use inner products $(\cdot, \cdot)$ to denote the $L^2$ inner product. When unspecified, it will be clear from context if we mean $L^2(\mathbb{R}_+)$ or $L^2(\Omega)$. When there is potential confusion, we will distinguish inner products on a one-dimensional surface (say $\{x = 0\}$) by writing $(\cdot, \cdot)_{x = 0}$. 

We will often use scaled differential operators: $\nabla_\eps := (\p_x, \sqrt{\eps}\p_y)$ and $\Delta_\eps := \p_{yy} + \eps \p_{xx}$. For functions $w: \mathbb{R}_+ \rightarrow \mathbb{R}$, we distinguish between $w'$ which means differentiation with respect to its argument versus $w_y$ which refers to differentiation with respect to $y$. 

Regarding unknowns, the central object of study in our paper are the remainders, $[u^{(\eps)}, v^{(\eps)}]$. By a standard homogenization argument (see subsection \ref{subsection.rem}), we may move the inhomogeneous boundary terms $a_i^{\eps}$ to the forcing and consider the homogeneous problem. We call the new unknowns $[u, v]$, and these are actually the objects we will analyze throughout the paper. 

When we write $a \lesssim b$, we mean there exists a number $C < \infty$ such that $a \le C b$, where $C$ is independent of small $L, \eps$ but could depend on $[u_s, v_s]$. We write $o_L(1)$ to refer to a constant that is bounded by some unspecified, perhaps small, power of $L$: that is, $a = o_L(1)$ if $|a| \le C L^\delta$ for some $\delta > 0$. 

We will, at various times, require localizations. All such localizations will be defined in terms of the following fixed $C^\infty$ cutoff function: 
\begin{align} \label{basic.cutoff}
\chi(y) := \begin{cases}1 \text{ on } y \in [0,1) \\ 0 \text{ on } y \in (2,\infty) \end{cases} \hspace{3 mm} \chi'(y) \le 0 \text{ for all } y > 0.  
\end{align}

We will use $\| \cdot \|_{loc}$ to mean localized $L^2$ norms. More specifically we take for concreteness $\| \cdot \|_{loc} := \| \cdot \chi(\frac{y}{10}) \|$.

We will define now the key norms that appear throughout our analysis:
\begin{definition} \label{defn.norms.intro} \normalfont Given a weight function $w = w(y)$, define:  
\begin{align} 
\begin{aligned} \label{defn.norms.ult}
&\| v \|_{X_w} := \eps^{-\frac{3}{16}}||||v||||_w + |||q|||_w, \\
&\| v \|_{Y_w} := ||||v||||_w + \sqrt{\eps}|||q|||_w, \\
&\| u^0 \|_{\Upsilon} := \| u^0_{yyy} \langle y \rangle \| + \| u^0_{ yy} \langle y \rangle \| + \| u^0_{ y} \| + \|u^0 \|_{loc}\\
&\| u^0 \|_{B} := \| u^0_\perp \|_{\Upsilon} +\eps^{\frac{1}{4}}|\kappa| + \|\eps^{\frac{1}{4}}u^0_{yy} \frac{\langle y \rangle}{\sqrt{v^1_e|_{x = 0}}} \|, \\
&\| u^0, v \|_{\mathcal{X}} := \|u^0 \|_B + \eps^{\frac{1}{4}} \| v \|_{X_1} + \eps^{\frac{1}{4}+}  \|v \|_{Y_1} + \eps^{\frac{1}{4}-} \| v \|_{Y_{w_0}}, \\
&|||q|||_w := \| \nabla_\eps q_x \cdot u_s w \| + \| \sqrt{u_s} \{ q_{yyy}, q_{xyy}, \sqrt{\eps} q_{xxy}, \eps q_{xxx} \} w \| + |q|_{\p, 2, w}\\
&||||v||||_w := \| \{\sqrt{\eps} v_{xyyy}, \eps v_{xxyy}, \eps^{\frac{3}{2}}v_{xxxy}, \eps^2 v_{xxxx} \} w \| + |v|_{\p, 3, w}\\
&| q |_{\p,2,w} :=  \|u_s q_{xy}w\|_{x = 0} + \|q_{xy} w\|_{y = 0} + \|\sqrt{\eps}u_s q_{xx}w\|_{x = L} + \|q_{yy}w\|_{y= 0} \\
&| v |_{\p, 3,w} := \|\eps^{\frac{3}{2}}\sqrt{u_s} v_{xxx} w\|_{x = 0} + \|\sqrt{\eps} u_s v_{xyy}w\|_{x = 0} + \|\eps u_s v_{xxy}w\|_{x = L}.
\end{aligned}
\end{align}
\end{definition}

\subsection{Overview of Proof} \label{subsection.Main}

Let us first recap the ideas introduced in \cite{GN}, which treated the case when the boundary $\{y = 0\}$ was moving with velocity $u_b > 0$. First, let us extract: 
\begin{align} \label{leading.1}
\text{Leading order operators in (\ref{eqn.vort.intro})} = - R[q] - u_{yyy}.
\end{align}

Due to the nonzero velocity at the $\{y = 0\}$ boundary, the quantity $\bar{u}|_{y = 0} > 0$. A central idea introduced by \cite{GN} is the coercivity of $R[q]$ over $\| \nabla_\eps q \|$.This coercivity relied on the fact that $q = \frac{v}{\bar{u}} = 1 \notin \text{Ker}(R)$, thanks to the non-zero boundary velocity of $\bar{u}|_{y = 0}$. Extensive efforts without success have been made to extracting coercivity from $R[q]$ in the present, motionless boundary, case. However, it appears that this procedure interacts poorly with the operator $\p_{yyy} u$, producing singularities too severe to handle.

\subsubsection*{\normalfont \textit{Part 0: The Central Objects}}

Our main idea is based on the observation that the $x$ derivative of (\ref{leading.1}) produces, at leading order:
\begin{align} \label{leading.2}
- \p_x R[q] + v_{yyyy}.
\end{align}

\noindent Unlike (\ref{leading.1}), these two operators enjoy better interaction properties. 

To this end, we split the equation (\ref{eqn.vort.intro}) into two pieces that are linked together. First, we take $\p_x$ of (\ref{eqn.vort.intro}) (call this ``DNS" for Derivative Navier-Stokes) to obtain: 
\begin{align}
\begin{aligned} \label{intro.v.sys}
&\text{DNS}(v) := - \p_x R[q] + \Delta_\eps^2 v +J(v) = - F_{u^0} + \eps^{N_0} \mathcal{N} + g_{(q)}, \\
&v|_{x = 0} = v_x|_{x = L} =  v_{xx}|_{x = 0} = v_{xxx} = 0. \\
&v|_{y = 0} = v|_{y = 0} = 0.  
\end{aligned}
\end{align}

\noindent Here, $F_{u^0}$ contains the $u^0$ dependencies, which arise through $u = u^0 - \int_0^x v_y$, and is defined as 
\begin{align} \label{U.def}
F_{u^0} := v_{sx} u^0_{yy} - u^0 \Delta_\eps v_{sx}.
\end{align}

\noindent $\mathcal{N}$ contain nonlinear terms and $g_{(q)}$ contains forcing terms, all of which are defined in (\ref{eqn.dif.1.app}). Note also the change in notation in (\ref{intro.v.sys}) as we have dropped the superscript $\eps$, and homogenized the boundary conditions on the sides $\{x = 0\}, \{x = L\}$. 

The second piece is to study the boundary trace, $u^0 = u|_{x = 0}$. By evaluating the vorticity equation (\ref{eqn.vort.intro}) at $\{x = 0\}$, we obtain the following system for $u^0$:
\begin{align}
\begin{aligned} \label{sys.u0.intro}
&\mathcal{L} u^0 := - u^0_{yyy} + v_s u^0_{yy} - u^0 \Delta_\eps v_s = F_{(v)} + g_{(u)}, \\
&F_{(v)} := -2\eps u_s u_{sx}q_x|_{x = 0} - 2\eps v_{xyy}|_{x = 0} - \eps^2 v_{xxx}|_{x = 0} + \eps v_s v_{xy}|_{x = 0}, \\
& u^0(0) = 0, u^0_y(\infty) = 0, u^0_{yy}(\infty) = 0.
\end{aligned}
\end{align}

\noindent As is evident, the right-hand side of (\ref{sys.u0.intro}) depends on (derivatives of) $v|_{x = 0}$. The term $g_{(u)}$ is a forcing term which is specified in (\ref{sys.u0.app}).

Thus, the approach we take is to analyze (\ref{intro.v.sys}) to control $v$ in terms of the boundary trace, $u^0$, and then to analyze (\ref{sys.u0.intro}) in order to control the boundary trace $u^0$ in terms of $v$. We may schematize this procedure via: 
\begin{align} \label{diag1}
u^0 \xrightarrow{\text{DNS}^{-1}} v \xrightarrow{\mathcal{L}^{-1}} u^0. 
\end{align}

\noindent We then recover a solution to the original Navier-Stokes equation (NS) via a fixed point of (\ref{diag1}). This structure of analysis gives rise to a linked set of inequalities (see below, (\ref{scheme.1}), for the scheme of estimates). 

\subsubsection*{ \normalfont \textit{Part 1: $\mathcal{L}^{-1}$ and Boundary Estimate of $u^0$}}

Let us turn now to the system, (\ref{sys.u0.intro}). We first decompose the coefficient $v_s$ (refer back to (\ref{exp.u}) for the definition) in two different ways: 
\begin{align}  \label{opt}
&v_s =\begin{cases} v_{\parallel} + \bar{v}^1_e + \sqrt{\eps} \bar{v}^1_p + \text{ higher order terms}, \\ v^0_p + v^1_e + \sqrt{\eps}v^1_p + \text{ higher order terms} \end{cases},
\end{align}

\noindent where $v_{\parallel}$ has been defined above in (\ref{parallel}), and $\bar{v}^1_e, \bar{v}^1_p$ are both defined in (\ref{intro.bar.prof}). The key point of these definitions can be gleaned by examining the leading order of $v_s$ which is $v^0_p + v^1_e$. Both of these quantities decay at $y = \infty$. In the first case of (\ref{opt}), we rewrite this sum as $\{v^0_p - v^0_p|_{y = 0}\} + \{v^1_e - v^1_e|_{Y = 0}\} = \bar{v}^1_p + \bar{v}^1_e$, where both of these quantities do not decay at $y = \infty$. Correspondingly, we have a decomposition of the operator $\mathcal{L}$ into: 
\begin{align} 
\begin{aligned} \label{intro.L.d}
&\mathcal{L} u^0 := \mathcal{L}_{\parallel} u^0 + \sqrt{\eps}A u^0 + \bar{v}^1_e u^0_{yy} + \text{ higher order terms}, \\
&\mathcal{L}_{\parallel} u^0 := - u^0_{yyy}+ v_{\parallel} u^0_{yy} - u^0 v_{\parallel yy}, \\
&A := \bar{v}^1_p u^0_{yy} -  u^0 \bar{v}^1_{pyy}.
\end{aligned}
\end{align}

It is first important to study the spectrum of $\mathcal{L}_{\parallel}$. Our first key observation is that $\bar{u}^0_p|_{x = 0}$ is an element in $\text{Ker}(\mathcal{L}_{\parallel})$ thanks to the Prandtl equation, (\ref{Pr.leading.intro}). Correspondingly, we decompose $u^0$ in the following manner:
\begin{definition} \normalfont Define ``parallel" profiles: 
\begin{align} \label{parallel}
u_{\parallel} = \bar{u}^0_p|_{x = 0} \text{ and } v_{\parallel} = \bar{v}^0_p|_{x = 0}, 
\end{align}

\noindent and the corresponding decomposition: 
\begin{align} \label{basic.dec}
&u^0 = u_{\perp} + \kappa u_{\parallel}, \hspace{5 mm} \kappa := \frac{ \omega[u^0]}{\omega[u_{\parallel}]}. 
\end{align}
\end{definition}

\noindent Here $\omega$ is a linear map satisfying $0 < \omega[u_{\parallel}] < \infty$. We refrain at this time from discussing the particular choice for $\omega$; this level of detail can be found in Section \ref{Section.1}. The reason for the use of ``parallel" and, correspondingly, ``perpendicular" is because $u_{\parallel}$ is in the kernel of the crucial operator, $\mathcal{L}_{\parallel}$ (see below, (\ref{intro.L.d})) whereas $u_\perp$ is orthogonal to the kernel.  

Our first estimate, (see Lemma \ref{L.estimate}), leads to the following lower bound:
\begin{align} \label{intro.first.est}
\| \mathcal{L}_{\parallel} u_{\perp} \langle y \rangle \| \gtrsim \| u_{\perp} \|_{\Upsilon} \text{ where } u_\perp \perp u_{\parallel}. 
\end{align}

\noindent The outcome of Lemma \ref{L.estimate} is then: 
\begin{align} \label{est.J.intro}
\| u_\perp \|_{\Upsilon}  \lesssim o(1) \eps^{\frac{1}{4}}|\kappa| + \| F_{(u)} \langle y \rangle \| + \bar{v}^1_e u_{\perp yy} \text{ contributions.}
\end{align}

Our second ingredient is to control the coefficient $\kappa$, the ``parallel" component of $u^0$, in the decomposition, (\ref{basic.dec}). For this, we use the equation (\ref{sys.u0.intro}). In particular, the operator $S u_{\parallel} = A u_{\parallel} + \frac{\bar{v}^1_e}{\sqrt{\eps}} u_{\parallel yy} - \sqrt{\eps} \Delta v^1_e u_{\parallel}$ is utilized to control this projection. To do this, we require the non-degeneracy and smallness condition in (\ref{charlie.1}) - (\ref{charlie.3}). Specifically, Lemma \ref{Lemma.kappa} gives
\begin{align} \label{intro.second.est}
\eps^{\frac{1}{4}} |\kappa| \lesssim  \| u_\perp \|_{\Upsilon}  +  \eps^{-\frac{1}{4}}\| F_{(u)} \langle y \rangle^{\frac{1}{2}+} \| + \bar{v}^1_{e} u^0_{yy} \text{ contributions}. 
\end{align}

Lastly, we need to show that $\bar{v}^1_e u^0_{yy}$ is a small perturbation. However, it is evident that $\bar{v}^1_e u^0_{yy}$ is an order 1 term. It is therefore difficult to imagine that these terms can be treated perturbatively. The key feature that we capitalize on is that $\bar{v}^1_e$ exhibits Euler scaling, and it thus suffices to localize to the far-field region: $Y = \frac{y}{\sqrt{\eps}} \gtrsim 1$. To capitalize on this localization, we employ the second decomposition, (\ref{opt}.2), under which most coefficients have decayed rapidly and become negligible in the region $y \gtrsim \frac{1}{\sqrt{\eps}}$. It is in this estimate that we demand $v^1_e|_{x = 0} > 0$ in order to extract a lower bound from $v^1_e u^0_{yy}$.

Combining the estimates (\ref{intro.first.est}),  (\ref{intro.second.est}) with Lemma \ref{lemma.Z.scale.1}, we are able to prove the following estimate 
\begin{align} \label{prop.intro.1.est}
\| u^0 \|_B & \lesssim \| F_{(u)} \{w_0 + \eps^{-\frac{1}{4}} \langle y \rangle^{\frac{1}{2}+} \}\| 
\end{align}

\subsubsection*{\normalfont \textit{Part 2: Solving DNS for $v$}}

We now turn our attention to (\ref{intro.v.sys}). The goal is to establish control over the norms $\| \cdot \|_{Y_{w_0}}, \| \cdot \|_{Y_1}, \| \cdot \|_{X_1}$. Since the DNS equation is the same as in \cite{GI1}, we simply state the following: 

\begin{proposition} \normalfont There exists a unique solution, $v$, to the system (\ref{intro.v.sys}) that satisfies (at the linear level): 
\begin{align}
\left.
\begin{aligned} \label{scheme.1}
&\| u^0 \|_{B}^2 \lesssim \eps^{\frac{1}{2}-} \| v \|_{Y_{1}}^2 + \eps^{\frac{1}{2}+} \| v \|_{Y_{w_0}}^2 + \eps^{\frac{1}{2}+\frac{3-}{16} } \| v \|_{X_1}^2 + \text{Data} \\
&\| v \|_{X_1}^2 \lesssim \eps^{-\frac{1}{2}} \|u^0 \|_{B}^2 + \text{Data} \\
&\| v \|_{Y_{1}}^2 \lesssim \eps^{\frac{3}{8}} \| v \|_{X_1}^2 + \eps^{\frac{3}{8}} \| u^0 \|_B^2+ \text{Data} \\
&\| v \|_{Y_{w_0}}^2 \lesssim o_L(1) \| v \|_{X_1}^2 + L \eps^{-\frac{1}{2}} \| u^0 \|_{B}^2 + \text{Data}.
\end{aligned}
\right\}.
\end{align}
\end{proposition}

\noindent Above $w_0$ is a specific weight, which for the purposes of the present discussion we will set to be $\frac{\langle y \rangle}{v^1_e|_{x = 0}}$. It is clear that the above scheme of estimates closes to yield control over $\| u^0, v \|_{\mathcal{X}}$.

\subsection{Other Works}

Let us now place this result in the context of the existing literature. To organize the discussion, we will focus on the setting of stationary flows in dimension 2. This setting in particular occupies a fundamental role in the theory, as it was the setting in which Prandtl first formulated and introduced the idea of boundary layers for Navier-Stokes flows in his seminal 1904 paper, \cite{Prandtl}. 

In this context, one fundamental problem is to establish the validity of the expansions (\ref{exp.u}). This was first achieved under the assumption of a moving boundary in \cite{GN} for $x \in [0,L]$, for $L$ sufficiently small. The method of \cite{GN} is to establish a positivity estimate to control $||\nabla_\eps v||_{L^2}$, which crucially used the assumed motion of the boundary. Several generalizations were obtained in \cite{Iyer}, \cite{Iyer2}, \cite{Iyer3}, all under the assumption of a moving boundary. First, \cite{Iyer} considered flows over a rotating disk, in which geometric effects were seen, \cite{Iyer2} considered flows globally in the tangential variable, and \cite{Iyer3} considered outer Euler flows that are non-shear.

The classical setup of a nonmoving boundary, considered by Prandtl, has remained open until recently.  We would like to highlight the exciting paper of \cite{Varet-Maekawa} as well as \cite{GI1} which both treat the no-slip boundary condition. These works are mutually exclusive. The main concern of \cite{GI1} treats the classical self-similar Blasius solution which appears to not be covered by \cite{Varet-Maekawa}. On the other hand, \cite{GI1} result does not cover a pure shear boundary layer of the form $(U_0(y), 0)$ since such shears are not a solution to the homogeneous Prandtl equation. 

For unsteady flows, expansions of the form (\ref{exp.u}) have been verified in the following works: \cite{Caflisch1}, \cite{Caflisch2}, \cite{DMM}, \cite{Mae}. The reader should also see \cite{Asano}, \cite{Taylor}, \cite{TWang}, \cite{Kelliher}, \cite{HLop}, \cite{BardosTiti} for related results. There have also been several works (\cite{GGN1}, \cite{GGN2}, \cite{GGN3}, \cite{GN2}, \cite{GrNg1}, \cite{GrNg2}, \cite{GrNg3}) establishing generic invalidity of expansions of the type (\ref{exp.u}) in Sobolev spaces in the unsteady setting. 

A related question is that of wellposedness of the Prandtl equation (the equation for $\bar{u}$, as defined in (\ref{intro.bar.prof})). Since this is not the concern of the present article, we very briefly list some works. In the stationary setting, we point the reader to \cite{Oleinik}, \cite{Oleinik1}, \cite{MD}. In the unsteady setting, the reader should consult \cite{AL}, \cite{MW}, \cite{KMVW}, \cite{Caflisch1} - \cite{Caflisch2},  \cite{Kuka}, \cite{Lom}, \cite{Vicol}, and \cite{GVM}, \cite{GVD}, \cite{GVN}), \cite{EE}, \cite{KVW}, \cite{Hunter} for wellposedness/ illposedness results and references therein. 

The above discussion is not comprehensive; we refer the reader to the review articles, \cite{E}, \cite{Temam} and references therein for a more thorough review of the wellposedness theory.

\section{$\mathcal{L}^{-1}$ and Boundary Estimates for $u^0$} \label{Section.1}

In this section, we study the quantity $u^0 = u|_{x = 0}$. 

\begin{remark} \normalfont For this section, we will work exclusively on the boundary $\{x = 0 \}$. Therefore, all functions (even those with natural extensions to all of $\Omega$) will be thought of as functions of $y$ only. Similarly, inner-products and norms will refer to functions defined on $\mathbb{R}_+$. We will therefore omit the notation $|_{x = 0}$. 
\end{remark}

We define the norm,
\begin{align} \label{L.norm}
\|h\|_{\Upsilon} := \|h_{yyy} \langle y \rangle\| + \|h_{yy}  \langle y \rangle \| + \| h_y\|_{loc} + \|h\|_{loc}.
\end{align}

\noindent We define the corresponding Sobolev space $\Upsilon$ via: 
\begin{align} \n
\Upsilon := \{ h \in L^2 : \| h \|_{\Upsilon} < \infty \}. 
\end{align}

\noindent A standard embedding shows that: 
\begin{align*}
\|h_y \| + \| h \langle y \rangle^{-1} \| \lesssim \| h \|_{\Upsilon}. 
\end{align*}

In this section, the equation we will analyze is:
\begin{align}
\left.
\begin{aligned} \label{sys.u0}
&\mathcal{L}_\delta u^0 := - u^0_{yyy} + (v_s + \delta) u^0_{yy} - u^0 \Delta_\eps v_s = F, \\
&F := \underbrace{-2\eps u_s u_{sx}q_x|_{x = 0} - 2\eps v_{xyy}|_{x = 0} - \eps^2 v_{xxx}|_{x = 0} - \eps v_s u_{xx}|_{x = 0}}_{F_{(v)}} + g_{(u)}, \\
& u^0(0) = 0, \p_y u^0(\infty) = 0, \p_{yy} u^0(\infty) = 0.
\end{aligned}
\right\}
\end{align}

 We are ultimately interested in the $\delta = 0$ case of (\ref{sys.u0}), which is (\ref{sys.u0.intro}). We will first use a decomposition of $v_s$ in order to decompose the operator $\mathcal{L}_\delta$. We first recall the definitions in (\ref{intro.bar.prof}) and (\ref{parallel}), and correspondingly decompose $v_s$ into: 
\begin{align*}
v_s = v_{\parallel} + \sqrt{\eps} \bar{v}^1_p + \sum_{i = 2}^n \sqrt{\eps}^i \bar{v}^i_p + \sum_{i = 1}^{n} \sqrt{\eps}^{i-1} \bar{v}^i_e. 
\end{align*}

\noindent We thus have a decomposition of: 
\begin{align*}
\mathcal{L}_\delta = \mathcal{L}_{\parallel} + \sqrt{\eps}A + J, 
\end{align*}

\noindent where
\begin{align} \label{label.L}
&\mathcal{L}_{\parallel} u^0 :=  -u^0_{yyy} +v_{\parallel}  u^0_{yy} - u^0  v_{\parallel yy}, \\ \label{label.A}
&Au^0 :=  \bar{v}^1_p u^0_{yy} - u^0 \bar{v}^1_{pyy}, \hspace{3 mm} r(y) :=  \bar{v}^1_p u_{\parallel y} - u_\parallel \bar{v}^1_{py}, \\ \label{label.J}
&J u^0 := \delta u^0_{yy} + u^0_{yy} \sum_{i = 1}^{n} \sqrt{\eps}^{i-1} \bar{v}^i_e + u^0_{yy} \sum_{i = 2}^n \sqrt{\eps}^i \bar{v}^i_p  - \eps u^0 \sum_{i = 0}^n \sqrt{\eps}^i v^i_{pxx} \\ \n
& \hspace{10 mm} - u^0 \sum_{i = 2}^n \sqrt{\eps}^i v^i_{pyy} - \eps u^0 \sum_{i = 1}^n \sqrt{\eps}^{i-1} \Delta v^i_e. 
\end{align}

Define the following bounded linear functional on $\Upsilon$: 
\begin{align*}
\omega[g] := (g_{yy}, u_{\parallel yy}), \hspace{5 mm} \omega: \Upsilon \rightarrow \mathbb{R}. 
\end{align*}

\noindent It is clear that $\omega$ is bounded on $\Upsilon$. As a result, we define: 
\begin{align*}
\Upsilon_\perp := \{ g \in \Upsilon : \omega[g] = 0 \}, 
\end{align*}

\noindent which is a closed subspace of $\Upsilon$. We will now project via: 
\begin{align} \label{basic.dec}
&u^0 = u_{\perp} + \kappa u_{\parallel}, \hspace{5 mm} \kappa := \frac{\omega[u^0]}{\omega[u_{\parallel}]}. 
\end{align}

\noindent  We note that $0 < \omega[u_{\parallel}] < \infty$. We now denote our $B$ norm by: 
\begin{align} \label{B.norm}
\| u^0 \|_{B} := & \| u_{\perp} \|_{\Upsilon} + \eps^{\frac{1}{4}} |\kappa| + \|\eps^{\frac{1}{4}}u^0_{yy} \frac{\langle y \rangle}{\sqrt{v_e}} \|.
\end{align}
 
\begin{theorem} \label{thm.Umain} \normalfont Assume the boundary data satisfy (\ref{charlie.1}) - (\ref{charlie.3}) and (\ref{assume.Euler.i}). Then there exists a unique solution to (\ref{sys.u0}) satisfying: 
\begin{align} \label{B.thm.estimate}
\|u^0 \|_{B} \lesssim \|F \{w_0 + \eps^{-\frac{1}{4}} \langle y \rangle^{\frac{1}{2}+} \}  \|.
\end{align} 
\end{theorem} 
 
This section will be devoted to establishing this theorem.

\subsection{Existence} \label{subsection.existence.u0}

We begin with an existence result. To state our existence result, we will \textit{assume the a-priori estimate} that we will establish in the forthcoming sections: 
\begin{align} \label{assume.prior}
\| u^0 \|_{B} \lesssim \|F w_1 \| \text{ for } u^0 \in B \text{ solutions to } (\ref{sys.u0}),
\end{align}

\noindent where $w_1 := w_0 + \eps^{-\frac{1}{4}} \langle y \rangle^{\frac{1}{2}+}$.

\subsubsection*{\normalfont \textit{Step 1: Highest Order Equation}}

We may start with the modified problem: 
\begin{align} \label{part1.ODE}
\begin{aligned}
&\Theta u := -u''' + \delta u'' = F \in C^\infty_0, \hspace{5 mm} u(0) = 0, \hspace{3 mm} \p_y^k u(\infty) = 0 \text{ for } k \ge 1. 
\end{aligned}
\end{align}

\begin{lemma} \normalfont Let $\delta > 0$. There exists a unique solution, $u$, to (\ref{part1.ODE}) whose derivatives vanishes outside of an interval $[0, I_0]$ for $I_0 < \infty$. Moreover, $u$ can be expressed explicitly in terms of $F$ via the formula $u = C_1 + u_p$, where $u_p$ is the particular solution associated to $F$, as defined below. 
\end{lemma}
\begin{proof} By assuming the solution $u = e^{ry}$, we obtain the system $-r^3 + \delta r^2 = 0$, which gives the following solutions to the homogeneous problem: $u_1 = 1, u_2 = y, u_3 = e^{\delta y}$. The task is now to produce a particular solution to the forcing, $F$. We may form the Wronskian matrix and its Forcing counterparts:
\begin{align*}
&\bold{W}(y) = 
\begin{bmatrix}
    1 & y & e^{\delta y}\\
    0 & 1 & \delta e^{\delta y} \\
    0 & 0 & \delta^2 e^{\delta y}
  \end{bmatrix} \hspace{10 mm} 
\bold{W}_1(y) = 
\begin{bmatrix}
    0 & y & e^{\delta y}\\
    0 & 1 & \delta e^{\delta y} \\
    F & 0 & \delta^2 e^{\delta y}
\end{bmatrix} \\
&\bold{W}_2(y) = 
\begin{bmatrix}
    1 & 0 & e^{\delta y}\\
    0 & 0 & \delta e^{\delta y} \\
    0 & F & \delta^2 e^{\delta y}
\end{bmatrix} \hspace{8 mm} 
\bold{W}_3(y) = 
\begin{bmatrix}
    1 & y & 0\\
    0 & 1 & 0\\
    0 & 0 & F
\end{bmatrix}
\end{align*}

The corresponding determinants are: 
\begin{align*}
&W(y) = |\bold{W}| = \delta^2 e^{\delta y}, \hspace{5 mm} W_1(y) = |\bold{W}_1| = F e^{\delta y} (\delta y-1) \\
&W_2(y) = |\bold{W}_2| = -\delta F e^{\delta y }, \hspace{2 mm} W_3(y) = |\bold{W}_3| = F.
\end{align*}

We now have the formula for the particular solution: $u_p = c_1 u_1 + c_2 u_2 + c_3 u_3, \hspace{3 mm} c_i' = \frac{W_i}{W}$. Solving gives: $c_1 = - \int_y^\infty F \delta^{-2} (\delta y' - 1) \ud y'$, $c_2 = \int_y^\infty \delta^{-1} F \ud y'$, $c_3 = - \int_y^\infty F \delta^{-2} e^{-\delta y'} \ud y'$. Recalling that $F$ is compactly supported, we thus have the boundary conditions: $\p_y^j u_p(\infty) = 0$ for all $j \ge 0, u_p(0) = c_1(0) + c_3(0)$. We will thus write the full solution as: $u = C_1 + C_2 y + C_3 e^{\delta y} + u_p$. We now evaluate at $y = \infty$ to see that $C_2 = C_3 = 0$. We evaluate at $y = 0$ to see that $C_1 = - u_p(0)$. This gives the full solution to (\ref{part1.ODE}), $u = C_1 + u_p$. Given this solution $u = C_1 + u_p = -u_p(0) + c_i u_i$, it is clear that the derivatives of $u$ are compactly supported since the coefficients $c_i$ are. 
\end{proof}

Fix $\delta > 0$. Define the normed space:
\begin{align*}
\| u \|_{\mathcal{T}_{m, n}} := \|u_{yyy} \langle y \rangle^m e^{ny}\| + \delta \|u_{yy} \langle y \rangle^m e^{ny}\| 
\end{align*}

\begin{lemma} \normalfont Fix $\delta > 0$. Let $F \in C^\infty_0(\mathbb{R})$. The solution, $u$ is in $\mathcal{T}_{m,n}$ and satisfies the following estimate: $\| u \|_{\mathcal{T}_{m,n}} \lesssim \| F \langle y \rangle^m e^{ny} \|$ for any $m, n$. 
\end{lemma}
\begin{proof}
We now square both sides of the equation (\ref{part1.ODE}) against the weight $\langle y \rangle^{2m} e^{2ny}$ to obtain 
\begin{align*}
|u''' \langle y \rangle^m e^{ny}|^2 + \delta^2 |u'' \langle y \rangle^{m} e^{ny}|^2 - 2\delta (u''', u'' \langle y \rangle^{2m} e^{2ny}) = \|F \langle y \rangle^m e^{ny}\|^2. 
\end{align*}

\noindent We integrate by parts the cross term, which is possible due to the compact support of $u$, and which generates the positive contribution 
\begin{align*}
2 \delta m \| u'' \langle y \rangle^{m-\frac{1}{2}} e^{ny} \|^2 + 2 \delta n \| u'' \langle y \rangle^{m} e^{ny} \|^2.
\end{align*}
\end{proof}

\subsubsection*{\normalfont \textit{Step 2: Arbitrary $F \langle y \rangle^m e^{ny} \in L^2$}}

We now remove the compact support hypothesis on $F$. 
\begin{lemma} \normalfont Fix $\delta > 0$. Let $F \langle y \rangle^m e^{ny} \in L^2$ for any $n, m$. Then there exists a unique solution $u \in \mathcal{T}_{m,n}$ to the system (\ref{part1.ODE}) that satisfies $\| u \|_{\mathcal{T}_{m,n}} \lesssim \| F \langle y \rangle^m e^{ny} \|$. 
\end{lemma}
\begin{proof} Given an arbitrary $F$ satisfying $\| F \langle y \rangle^m e^{ny} \| < \infty$, there exists $\tilde{F}_j \in C^\infty_0$ such that $\tilde{F}_j \rightarrow F  \langle y \rangle^{m} e^{ny}$ in $L^2$. Define $F_j = \frac{1}{\langle y \rangle^m e^{ny}} \tilde{F}_j$ so that $F_j \langle y \rangle^{m} e^{ny} \rightarrow F \langle y \rangle^m e^{ny}$ in $L^2$ and clearly $F_j \in C^\infty_0$ as well. We define $u_j$ as solutions to $\Theta u_j = F_j$. In this case, we have $\| u_j \|_{\mathcal{T}_{m,n}} \lesssim \| F_j \langle y \rangle^m e^{ny}\| \le \|F \langle y \rangle^m e^{ny}\|$ by the previous lemma. Thus, there exists a $\mathcal{T}_{m,n}$-weak$^\ast$ limit, called $u$. We now multiply by a compactly supported test function, $\phi$, and integrate: $(u_j'', \phi') + ( v_s u_j'', \phi) - ( \Delta_\eps v_s u_j, \phi) = ( F_j, \phi)$. It is clear we can pass to the limit to obtain a strong solution $u$, which moreover enjoys the estimate $\| u \|_{\mathcal{T}_{m,n}} \lesssim \| F \langle y \rangle^m e^{ny} \|$.
\end{proof}

\subsubsection*{\normalfont \textit{Step 3: Compact Perturbations}}

\begin{lemma} \normalfont Fix $\delta > 0$. Let $F w_0 \in L^2$. Then there exists a unique solution, $u$, satisfying $\mathcal{L}_\delta u = F$, $u(0) = 0$, $\p_y^l u(\infty) = 0$ for $l \ge 1$ which satisfies the bound $\| u \|_{B} \lesssim \|F w_1 \|$.
\end{lemma}
\begin{proof} First, express $\mathcal{L}_\delta u = \Theta u + Ku = F$, where $K = v_s \p_y^2 - \Delta_\eps v_s$. Next, it is straightforward to see $\Theta^{-1}K: H^2 \rightarrow H^2$ is a compact operator. Thus, by the Fredholm alternative we must exhibit uniqueness of the homogeneous problem, that is with $F = 0$. Assume $u \in H^2$ is a solution to $\Theta u = - Ku$. Then by applying the previous lemma, $u \in \mathcal{T}_{m,n}$ for any $m, n$. This in particular implies that $u \in B_{\gamma}$. Our assumed \textit{a-priori} estimate, (\ref{assume.prior}), may then be used to yield uniqueness of the homogeneous solution and also the bound $\|u\|_{B} \lesssim \|F w_1 \|$. 
\end{proof}

\subsubsection*{\normalfont \textit{Step 4: $\delta \downarrow 0$}}

Our estimates are uniform in $\delta$, and thus we can subsequently take the limit as $\delta \downarrow 0$ in the standard manner. We have thus established by combining Steps 1 - 4: 

\begin{proposition}[Existence] \label{prop.ex.2} \normalfont Let $Fw_1 \in L^2$, and let $\delta \ge 0$. Assume the \textit{a-priori} estimate, (\ref{assume.prior}) holds. Then there exists a unique solution to $\mathcal{L}_\delta u^0 = F$, $u^0(0) = \p_y^{l} u^0(\infty) = 0$ for $l \ge 1$ which satisfies the bound $\| u^0\|_{B} \lesssim \|F w_1 \|$. 
\end{proposition}

\subsection{Coercivity of $\mathcal{L}_{\parallel}$}

The main proposition here is: 
\begin{proposition} \label{prop.Lvs} \normalfont Let $u_\perp$ be defined as in (\ref{basic.dec}). Then: 
\begin{align} \label{propLbig}
\| \mathcal{L}_{\parallel} u_\perp \cdot \langle y \rangle \| \gtrsim \|u_\perp \|_{\Upsilon}. 
\end{align}
\end{proposition}  

\begin{lemma} \label{propn.span} \normalfont Elements of the three dimensional kernel of $\mathcal{L}_{\parallel}$ can be written as the following linear combination: $c_1 u_{\parallel} + c_2 \tilde{u}_s + c u^p$,  where $c_1, c_2, c \in \mathbb{R}$. Here: 
\begin{align*}
&\tilde{u}_s := u_{\parallel} \int_1^y \frac{u_{\parallel}(1)^2}{u_{\parallel}^2} \exp \Big[ \int_1^z v_{\parallel} \ud w \Big] \ud z, \\
&u^p :=  \tilde{u}_s \int_0^y u_{\parallel} \exp \Big[ -\int_1^z v_{\parallel} \Big]  -u_{\parallel} \int_0^y \tilde{u}_s \exp \Big[ - \int_1^z v_{\parallel} \Big].
\end{align*}
\end{lemma}
\begin{proof} Integrating up once, we define the operator $\mathcal{L}_{\parallel}^1 := -u_{yy} + v_{\parallel} u_y - u   v_{\parallel y}$. One solution to the homogeneous equation, $\mathcal{L}_{\parallel}^1 u = 0$, is $u_{\parallel}$. By supposing the second spanning solution is of the form $\tilde{u}_s := u_{\parallel} a(y)$, we may derive the equation: $a''(y) = \Big[ v_{\parallel} - 2\frac{u^0_{\parallel y}}{u_{\parallel}} \Big] a'(y)$. Solving this equation gives one solution:
\begin{align} \label{a.def}
a'(y) = \frac{u_{\parallel}(1)^2}{u_{\parallel}^2} \exp \Big[ \int_1^y v_{\parallel} \Big], \hspace{3 mm} a(y) = \int_1^y \frac{u_{\parallel}(1)^2}{u_{\parallel}^2} \exp \Big[ \int_1^z v_{\parallel} \ud w\Big] \ud z. 
\end{align}
 
\end{proof}

We shall need asymptotic information about $\tilde{u}_s$:
\begin{lemma} \label{lemma.asy.1} \normalfont As defined in Lemma \ref{propn.span}, $\tilde{u}_s$ satisfies the following asymptotics: 
\begin{align} 
\begin{aligned}
&\tilde{u}_s|_{y = 0} \sim -1 \text{ and }\tilde{u}_{sy} \sim 1 \text{ as } y \downarrow 0, \\
& \tilde{u}_{sy},  \tilde{u}_{sy}, \tilde{u}_s  \sim \exp[v_\parallel(\infty) y] \text{ as } y \uparrow \infty. 
\end{aligned}
\end{align}
\end{lemma}
\begin{proof} For convenience, denote 
\begin{align*}
g(y) = \exp [ \int_1^y v_{\parallel} ]. 
\end{align*}

\noindent  By rewriting $v_{\parallel} = v_{\parallel}(\infty) + [v_{\parallel} - v_{\parallel}(\infty)]$, and using that the latter difference decays rapidly, we obtain the basic asymptotics $g \sim \exp[ v_{\parallel}(\infty) y]$ as $y \uparrow \infty$. An expansion of $a$, given in (\ref{a.def}), near $y = 0$ gives $a(y) \approx \int_1^y \frac{1}{z^2} \ud z \sim -\frac{1}{z}|_{1}^y  = 1-\frac{1}{y}$. Thus: $\tilde{u}_s|_{y = 0} \sim u_{\parallel} [1 - \frac{1}{y} ] \sim -1$. At $y = \infty$, we have the asymptotics:
\begin{align*}
 \tilde{u}_s = u_\parallel \int_1^y \frac{u_\parallel(1)^2}{u_\parallel^2} g(z)  \sim \int_1^y \exp[v_{\parallel}^\infty  y] \ud z \sim \exp[v_{\parallel}^\infty y]. 
\end{align*}

\noindent  We now differentiate to obtain 
\begin{align*}
\tilde{u}_{sy} = u_{\parallel y}a + u_\parallel a'(y) = u_{\parallel y}a(y) + \frac{u_\parallel (1)^2}{u_\parallel} g(y) \sim \exp[v_{\parallel}^\infty y].
\end{align*}

\noindent  To evaluate $\tilde{u}_{sy}$ at $y = 0$, we need more precision. Expansions give: 
\begin{align*}
&u_{\parallel} a'(y) = \frac{u_\parallel(1)^2}{u_s} g(y) \sim \frac{u_\parallel(1)^2}{u_{\parallel y}(0)y} g(y) \text{ for }y \sim 0, \text{ and } \\
&u_{\parallel y}a(y) \sim u_{\parallel y}(0) u_\parallel(1)^2 \int_1^y \frac{1}{u_\parallel^2} g(z)  \sim \frac{u_\parallel(1)^2}{u_{\parallel y}(0)} \int_1^y \frac{g(z)}{z^2} \ud z. 
\end{align*}

\noindent  We have used the fact that $\frac{1}{u_\parallel^2}$ does not contribute a factor of $\frac{1}{z}$ following the singularity of $\frac{1}{z^2}$. Indeed, Taylor expanding, using that $u_{\parallel yy}(0) = u_{\parallel yyy}(0) = 0$ (see the first identity in (\ref{prof.pick})), and the elementary identity for any $a, b \in \mathbb{R}$, $\frac{1}{a-b} - \frac{1}{a} = \frac{b}{a(a-b)}$, one obtains: 
\begin{align*}
\frac{1}{u_\parallel(z)^2} =  \frac{1}{u_{\parallel y}(0) z^2} + \bigO(z). 
\end{align*}

It remains to show $\int_1^y \frac{g(z)}{z^2} \ud z \sim -\frac{g(y)}{y}$. We decompose the integral into region $[1, y_\ast]$ and $[y_\ast, y]$ for $0 < y \le z \le y_\ast$. The $[1, y_\ast]$ integral contributes an $\bigO(1)$ constant. In the $[y_\ast, y]$  region, the Taylor expansion is valid: 
\begin{align*}
\int_{y_\ast}^y \frac{g(z)}{z^2} \ud z \sim \int_{y_\ast}^y \frac{g(y)}{z^2} \ud z + g'(y)\int_{y_\ast}^y \frac{z-y}{z^2} \ud z \sim g(y) [\frac{1}{y_\ast} - \frac{1}{y} ] + g'(y) \phi(y), 
\end{align*}

\noindent  where $| \phi(y)| \lesssim |\log y|$. We now use that $v_{\parallel}(0) = v_{\parallel}'(0) = 0$ and $g'(y) = v_\parallel (y) g(y)$ to show that $g'(y) \sim y^2$. Thus, $g'(y) \phi(y)$ vanishes as $y \rightarrow 0$. We thus have verified that $I(y) \sim -\frac{g(y)}{y}$. 

We now compute two derivatives: 
\begin{align*}
\tilde{u}_{syy} &= u_{\parallel yy}a + 2u_{\parallel y}a'(y) + u_\parallel a''(y) \\
& \sim a''(y) \sim \p_{y} \{ \frac{1}{u_\parallel^2} \exp[v_\parallel^\infty y] \} \sim \exp[v_{\parallel}^\infty y] \text{ as } y \uparrow \infty.
\end{align*}
\end{proof}

We now establish a general compactness lemma:
\begin{lemma} \label{compactness.lemma} \normalfont Let $H_D, H$ be Hilbert spaces, with inner-product $(\cdot , \cdot)_D$ and $(\cdot, \cdot)$ and associated norms, $\|\cdot\|_D$, $\| \cdot \|$. Let $T: H_D \rightarrow H$ be a bounded operator that is bounded below: $\|T x \| \ge C_1 \|x \|_D$. Let $T_c: H_D \rightarrow H$ be a compact operator. Assume $\text{Ker}(T+T_c) = \{0\}$. Then there exists a $C_0 > 0$ such that 
\begin{align*}
\|\{T + T_c\} x \| \ge C_0 (\|T x \| + \| T_c x \|).
\end{align*}
\end{lemma}

\begin{proof} We will assume that $\text{Ker}(T+T_c) = \{0\}$, and assume that the lower bound of the lemma does not hold. In this case, for each $k > 0$, we may find a sequence $\{x_n\} \in H$ which satisfies: 
\begin{align} \label{asu.1}
&\|Tx_n\|^2 + \|T_c x_n \|^2 = 1 \text{ and } \| \{T + T_c\}  x_n \|^2 = \frac{1}{n}.
\end{align}

\noindent  As $T$ is bounded below, $\{ x_n \}$ itself is a bounded sequence in $H_D$, and thus there exists a limit point, $x_\ast \in H_D$ such that, upon passing to a subsequence and immediately reindexing, $x_n \rightharpoonup x_\ast$ weakly in $H_D$. This also implies that $Tx_n \rightharpoonup T x_\ast$ and $\| T_c x_n - T_c x_\ast \| \rightarrow 0$. 

We will establish $\|\{T + T_c\} x_\ast\| = 0$, but $\|T x_\ast \| + \| T_c x_\ast \| = 1$, thereby contradicting the triviality of $\text{Ker}(T+T_c)$. For the first step, we combine weak lower-semicontinuity with a polarization argument. 
\begin{align} \n
0 \le \|T x_\ast +T_c x_\ast \|^2 &= \|T x_\ast\|^2 + \|T_c x_\ast \|^2 + 2(T x_\ast, T_c x_\ast) \\ \n
& \le \|T x_n \|^2 +\| T_c x_n \|^2 + 2(T x_n, T_c x_n) + \mathcal{P}_n \\ \n
& = \|T x_n + T_c x_n \|^2 + \mathcal{P}_n \\ \n
& \le \frac{1}{n} + |\mathcal{P}_n|.
\end{align}

\noindent  Here, $\mathcal{P}_n := (T x_n, T_c x_\ast - T_c x_n) + (T x_\ast - T x_n, T_c x_\ast)$ is a result of polarizing, and vanishes as $n \rightarrow \infty$: 
\begin{align*}
(T x_\ast, T_c x_\ast) = &(Tx_n, T_c x_\ast) + (T x_\ast - T x_n, T_c x_\ast) \\
= & (T x_n, T_c x_n) + (T x_n, T_c x_\ast - T_c x_n) + (T x_\ast - T x_n, T_cx_\ast) \\
= & (T x_n , T_c x_n) + \mathcal{P}_n,
\end{align*}

\noindent For the first term in $\mathcal{P}_k$, we use strong convergence of $\{T_c x_n\}$: 
\begin{align*}
|(T x_n, T_c x_\ast - T_c x_n)| \le \| T x_n \| \| T_c x_\ast - T_c x_n \| \rightarrow 0. 
\end{align*}

\noindent For the second term in $\mathcal{P}_k$, we use weak convergence of $\{T x_n \}$: 
\begin{align*}
(Tx_\ast - T x_n, T_c x_\ast) \rightarrow 0. 
\end{align*}

\noindent Thus $\|\{T + T_c\} x_\ast \| = 0$. 

By squaring, we see that to establish $\|T x_\ast\|^2 + \|T_c x_\ast \|^2 = 1$, it suffices to establish $-2(T x_\ast, T_c x_\ast) = 1$. The two equations in (\ref{asu.1}) imply that $-2(T x_n, T_c x_n) = 1 + o(\frac{1}{n})$. Again a polarization argument establishes that we may pass to the limit in the cross term to obtain: $-2(T x_n, T_c x_n) \rightarrow -2(Tx_\ast, T_c x_\ast)$. Applying the implication of the lemma to $x_\ast$ gives the desired contradiction. 

\end{proof}

We will now apply Lemma \ref{compactness.lemma} twice, starting with:

\begin{lemma} \normalfont \label{lemma.real.pos} Assume $\|h \|_{\Upsilon} < \infty$. Then: 
\begin{align*}
\| [- h_{yyy} + v_{\parallel} h_{yy}]  \langle y \rangle \|^2 \gtrsim \|h_{yyy} \langle y \rangle\|^2 + \|h_{yy} \langle y \rangle\|^2.
\end{align*} 
\end{lemma}
\begin{proof} Define $H = L^2(\langle y \rangle)$ and $H_D$ via the norm: $\| h \|_{D} := \|h_{yyy} \langle y \rangle \| + \|h_{yy} \langle y \rangle \|$. We take $Th = - h_{yyy} + v_{\parallel}^\infty h_{yy}$ and $T_ch = (v_{\parallel} - v_{\parallel}^\infty) h_{yy}$. A standard argument shows that $T_c(H_D) \hookrightarrow \hookrightarrow L^2(\langle y \rangle)$ and thus $T_c$ is a compact operator. Thus to establish the inequality, we must rule out elements of $\text{Ker}(T+T_c)$, which are zero solutions to $h''' = v_\parallel h''$. An explicit integration yields all solutions are of the form $h'' = C_0 \exp[ \int_0^y v_\parallel]$. Since $v_\parallel \ge 0$, this is in contradiction to $h''(\infty) = 0$. To conclude, we establish positivity of $T$, which follows by a straightforward integration by parts: 
\begin{align*}
\|Th \langle y \rangle \|^2 = & \| h_{yyy} \langle y \rangle \|^2 + |v_\parallel^\infty|^2 \| h_{yy} \langle y \rangle \|^2 - 2(h_{yyy}, v_{\parallel}^\infty h_{yy} \langle y \rangle^2) \\
= & \| h_{yyy} \langle y \rangle \|^2 + |v_\parallel^\infty|^2 \| h_{yy} \langle y \rangle \|^2 + 2 v_\parallel^\infty \| h_{yy} \langle y \rangle^{\frac{1}{2}} \|^2 + h_{yy}^2 v^\infty_{\parallel}|_{y = 0}.
\end{align*}
\end{proof}

\begin{lemma} \label{Lemma.cpct.1} \normalfont Let $u_\perp \in \Upsilon_\perp$. The following coercivity estimate holds:
\begin{align} \label{cont.1}
\|\mathcal{L}_{\parallel}[u_\perp] \cdot \langle y \rangle\|^2 \gtrsim \|[- u_{\perp yyy} + v_{\parallel} u_{\perp yy}] \langle y \rangle \|^2 + \|u_\perp v_{\parallel yy} \langle y \rangle\|^2.
\end{align}
\end{lemma}

\begin{proof}  Here, we define $T = - h_{yyy} + v_\parallel h_{yy}$, and $T_c = h v_{\parallel yy}$. We take $H_D = \Upsilon_\perp$. That $T$ is bounded below follows from the previous lemma. We proceed to rule out elements of $\text{Ker}(T+T_c) \subset \Upsilon_\perp$. An appeal to Proposition \ref{propn.span} shows that a general solution for the operator $L_{v_s}h= 0$ can be expressed as: $h = c_1 u_{\parallel} + c_2 \tilde{u}_s + c u^p$. As $u_{\parallel}(0) = u^p(0) = 0$, the boundary condition $h(0) = 0$ eliminates $c_2$. We next use $h_{yy}(\infty) = 0$ to eliminate $c$. In particular, $h_{yy} = c_1 u_{\parallel}'' + cu^p_{yy}$ is computed via: 
\begin{align} \label{cont.exp}
\lim_{y \rightarrow \infty} u^p_{yy} = & \lim_{y \rightarrow \infty} \Big\{ \tilde{u}_{syy}\int_0^y u_{\parallel} \exp \Big[ - \int_1^z v_{\parallel} \Big] + \tilde{u}_{sy} u_{\parallel} \exp \Big[ - \int_1^y v_{\parallel} \Big] \Big\}.
\end{align}

\noindent  The first term is dominant, thereby preventing cancellation between these two factors. To see this, for the first term, as $v_\parallel \ge 0$ and $\lim_{y \rightarrow \infty} v_{\parallel} \ge c_0 > 0$, we have, for large enough $y$, that the integral $\int_0^y u_s \exp [ - \int_1^z v_\parallel ] \ud z \gtrsim 1$. Coupling this elementary bound with Lemma \ref{lemma.asy.1} yields: 
\begin{align*}
 \tilde{u}_{syy} \int_0^y u_\parallel \exp [ - \int_1^z v_\parallel ] \ud z \sim \tilde{u}_{syy} \sim \exp[v_\parallel(\infty) y] \text{ as } y \uparrow \infty.
\end{align*}

On the other hand, Lemma \ref{lemma.asy.1} also gives: 
\begin{align*}
\Big|\tilde{u}_{sy} u_\parallel \exp [ - \int_1^y v_\parallel ] \ud z\Big| \lesssim \exp[v_{\parallel}(\infty) y] \exp[-v_{\parallel}(\infty) y] \lesssim 1. 
\end{align*}

\noindent  Thus, (\ref{cont.exp}) contributes exponentially at $y = \infty$, and we must conclude $c = 0$. Finally, we recall (\ref{propLbig}), and we use the condition $\omega[h] = 0$ for elements $h \in \Upsilon_\perp$ to conclude that $c_1 = 0$.

\end{proof}

\subsection{$B$ Estimate} \label{subsection.Bgamma}

We begin with the following lemma: 

\begin{lemma}\normalfont If $h \in \mathcal{L}_{\parallel}(\Upsilon)$ then $\bold{d}(h) = 0$.  
\end{lemma}
\begin{proof} Assume that $h \in \text{Range}(\mathcal{L}_{\parallel})$. Then there exists an element, $\tilde{u} \in \Upsilon$, such that $\mathcal{L}_{\parallel}(\tilde{u}) =h$. By decomposing $\tilde{u} = \bar{u}_\perp + m_2 u_{\parallel}$, we see that $L_{\parallel}[\bar{u}_\perp] = h$. Generically, such a solution, $\bar{u}_\perp$, is of the form:  
\begin{align*}
\bar{u_\perp} = \alpha u_{\parallel} + \alpha_2 \tilde{u}_s + C_0 u^p - u^p_{(h)}, 
\end{align*}

\noindent where 
\begin{align} \n
u^p_{(h)} :=  \tilde{u}_s \int_0^y u_{\parallel} \exp\Big[ -\int_1^y v_{\parallel} \Big] I_y[h] -u_{\parallel} \int_0^y \tilde{u}_s \exp \Big[ - \int_1^y v_{\parallel} \Big] I_y[h].
\end{align}

\noindent Above, we note that $I_y[h]$ is well defined as $h \in L^2(\langle y \rangle) \hookrightarrow L^1$. 

First, we use that $\bar{u_\perp}(0) = 0$ to eliminate the coefficient in front of $\tilde{u}_s$. Next, we take the indefinite integral of $\mathcal{L}_{\parallel} \bar{u}_\perp$ and $h$ (thereby creating the constant $C_0$ below) to obtain: 
\begin{align} \label{Int.eqn.1}
- \bar{u}_{\perp yy} + v_\parallel \bar{u}_{\perp y} - v_{\parallel y} \bar{u}_\perp = C_0 +  I_y[h]. 
\end{align}

\noindent  Evaluating this equation as $y \uparrow \infty$ shows that $C_0 = 0$. We have thus determined that $\bar{u_\perp} = \alpha u_{\parallel} - u^p_{(r)}$. We now compute $\bar{u_\perp}'$ directly using the expression for $u^p_{(h)}$, which yields: 
\begin{align*}
\lim_{y \uparrow \infty} \bar{u_\perp}'(y) =& \lim_{y \uparrow \infty} \tilde{u}_{sy} \Big\{ -  \int_0^y I_y[h] K  \Big\} = 0,
\end{align*}

\noindent  which implies that $\bold{d}(h) = 0$. 
\end{proof}

\begin{lemma} \label{L.estimate} \normalfont  Assume (\ref{charlie.1}) - (\ref{charlie.3}). Let $u^0 = u_\perp + \kappa u_{\parallel}$ be a solution to (\ref{sys.u0}). Then the following estimate holds:
\begin{align} \label{estimate.J}
\| u_\perp \|_{\Upsilon}  \lesssim \delta_s\eps^{\frac{1}{4}}|\kappa| + \| u_{\perp yy} \langle y \rangle \{1 - \chi(\frac{Y}{\gamma}) \} \| + \| F \langle y \rangle \|. 
\end{align}
\end{lemma}
\begin{proof} We will decompose the operator $\mathcal{L}_\delta$ in the following manner: 
\begin{align} \n
\mathcal{L}_\delta u^0 = & \mathcal{L}_{\parallel} u^0 + \sqrt{\eps} A u^0 + \bar{v}^1_e u^0_{yy} - \eps \Delta v^1_e u^0 + \tilde{J}u^0 \\ \n
= & \mathcal{L}_{\parallel} u_\perp + \sqrt{\eps} \kappa A u_{\parallel} + \kappa \bar{v}^1_e u_{\parallel yy} - \kappa \eps \Delta v^1_e u_{\parallel} \\ \n
& + \sqrt{\eps} A u_\perp + \bar{v}^1_e u_{\perp yy} - \eps \Delta v^1_e u_{\perp} + \tilde{J}u^0 \\ \label{start}
= & \mathcal{L}_{\parallel} u_\perp + \sqrt{\eps} \kappa S(u_{\parallel}) + \sqrt{\eps} S(u_\perp) + \tilde{J}u^0, 
\end{align}

\noindent where $\tilde{J}u^0 = Ju^0 - \{ \bar{v}^1_e u^0_{yy} - \eps \Delta v^1_e u^0 \}$ is explicitly given by
\begin{align*}
\tilde{J} :=& \delta u^0_{yy} + u^0_{yy} \sum_{i = 2}^n \sqrt{\eps}^i \bar{v}^i_e + u^0_{yy} \sum_{i = 2}^n \sqrt{\eps}^i \bar{v}^i_p - \eps u^0 \sum_{i = 0}^n \sqrt{\eps}^i v^i_{pxx} \\
& - u^0 \sum_{i = 2}^n \sqrt{\eps}^i v^i_{pyy} - \eps u^0 \sum_{i = 2}^n \sqrt{\eps}^{i-1} \Delta v^i_e,
\end{align*}

\noindent and we have defined $S$ via 
\begin{align*}
Sg :=  \sqrt{\eps} A g + \bar{v}^1_e g_{yy} - \eps \Delta v^1_e g.
\end{align*}

We use (\ref{start}) to write the equation as 
\begin{align} \label{bstatech1}
\mathcal{L}_{\parallel} u_\perp = F - \sqrt{\eps} \kappa S(u_{\parallel}) - \sqrt{\eps} S(u_\perp) - \tilde{J} u^0. 
\end{align}

We place both sides of the equation in $L^2(\langle y \rangle)$. On the left-hand side of (\ref{bstatech1}), we produce coercivity over $\| u_\perp \|_{\Upsilon}$. Next, we have
\begin{align*}
\| \sqrt{\eps}\kappa S(u_\parallel) \langle y \rangle \| \le & \| \sqrt{\eps} \kappa A u_\parallel \langle y \rangle \| + \| \kappa \bar{v}^1_e u_{\parallel yy} \langle y \rangle \| + \|\kappa \eps \Delta v^1_e u_\parallel \langle y \rangle \| \\
\le & \sqrt{\eps}|\kappa| \| A u_\parallel \langle y \rangle \| + \sqrt{\eps}\kappa | \bar{v}^1_{eY}(0) \langle y \rangle^2 u_{\parallel yy} \|_2 \\
& + \eps^{\frac{1}{4}} \delta_s |\kappa| \| \frac{u^0_{eYY}}{u^0_e} Y \|_{L^2_Y} \\
\lesssim & ( \delta_s \eps^{\frac{1}{4}} + \sqrt{\eps} ) |\kappa|.
\end{align*}

We treat the terms in $\sqrt{\eps}S(u_\perp)$ perturbatively. First, 
\begin{align*}
\| \sqrt{\eps} A u_\perp \langle y \rangle \| \lesssim \sqrt{\eps} \| u_\perp \|_{\Upsilon}.
\end{align*}

Next, we treat the $\bar{v}^i_e = v^i_e - v^i_e|_{y = 0}$ term by localizing. Recalling the definition of $\chi$ in (\ref{basic.cutoff}), we will fix $\gamma > 0$ to be a small parameter, and then split: 
\begin{align*}
\|\bar{v}^i_e u_{\perp yy} \langle y \rangle\| \le \|\bar{v}^i_e u_{\perp yy} \langle y \rangle \chi(\frac{Y}{\gamma})\| +  \|\bar{v}^i_e u_{\perp yy} \langle y \rangle \{1- \chi(\frac{Y}{\gamma})\}\|. 
\end{align*}

\noindent The localized contribution is estimated upon using $|\bar{v}^1_e| \lesssim Y$ via: 
\begin{align*}
\| \bar{v}^1_e u_{\perp yy} \chi(\frac{Y}{\gamma}) \langle y \rangle \| \le \| \bar{v}^1_e \chi(\frac{Y}{\gamma}) \|_\infty \gamma \| u_{\perp yy} \langle y \rangle \| \lesssim \gamma \| u_{\perp yy} \langle y \rangle \|. 
\end{align*}

\noindent As $\gamma << 1$, we absorb this into the $\|u_\perp \|_{\Upsilon}$ term appearing on the left-hand side. The nonlocal component contributes to the right-hand side of estimate (\ref{estimate.J}). 

Next, for the shear contribution we use that $\eps \langle y \rangle^2 \lesssim \langle Y \rangle^2$ to estimate:
\begin{align*}
\| \eps \Delta v^1_e u_\perp \langle y \rangle \| &= \| \eps \Delta v^1_e \frac{u_\perp}{\langle y \rangle} \langle y \rangle^2 \| \\
& \le \| \Delta v^1_e \langle Y \rangle^2 \|_\infty \| u_{\perp y} \| = o(1) \| u_\perp \|_{\Upsilon},
\end{align*}

\noindent which is absorbed to the left-hand side thanks to (\ref{charlie.1}) - (\ref{charlie.3}) and since the Euler equation for $v^1_e$ reads: 
\begin{align*}
\Delta v^1_e = \frac{u^0_{eYY}}{u^0_e}v^1_e. 
\end{align*}

We now move to the terms in $\tilde{J}$, which are all treated perturbatively. First, using $\delta < \sqrt{\eps}$: 
\begin{align*}
\|\delta u^0_{yy} \langle y \rangle\| \le \delta \|u_{\perp yy} \langle y \rangle\| + \frac{\delta}{\sqrt{\eps}} |\sqrt{\eps}\kappa| \cdot \|u_{\parallel yy} \langle y \rangle \|.
\end{align*}

The $\bar{v}^i_e u^0_{\perp yy}$ terms are treated the same as in the $\bar{v}^1_e u^0_{\perp yy}$ term. For the parallel contribution, we estimate for $i \ge 2$:   
\begin{align*}
\| \sqrt{\eps}^{i-1} \bar{v}^i_e u_{\parallel yy} \kappa \langle y \rangle \| \le & \sqrt{\eps}^{i-1} \Big[ \| \bar{v}^i_e u_{\parallel yy} \kappa \langle y \rangle \chi(Y) \| +  \| \bar{v}^i_e u_{\parallel yy} \kappa \langle y \rangle \{1 - \chi(Y) \} \| \Big] \\
\lesssim & \eps |\kappa|.
\end{align*}

The final $u^0_{yy}$ term may be estimated directly as:
\begin{align*}
\|u^0_{yy} \sum_{i = 2}^n \sqrt{\eps}^i \bar{v}^i_p \langle y \rangle\| \lesssim \eps \|u_{\perp yy} \langle y \rangle\| + \sqrt{\eps} \|u_{\parallel yy} \langle y \rangle\|.  
\end{align*}

We now move to the $u^0$ terms from $\tilde{J}$, (\ref{label.J}). The rapid decay of $v^0_p$ gives: 
\begin{align*}
\|\eps v^0_{pxx} u^0 \langle y \rangle\| \le & \|\eps v^0_{pxx}[u_\perp +\kappa u_\parallel] \langle y \rangle\| \\
\le & \eps \| v^0_{pxx} \langle y \rangle^{2} \|_\infty  \|u_\perp \langle y \rangle^{-1}\| + \sqrt{\eps} |\sqrt{\eps}\kappa| \| v^0_{pxx} \langle y \rangle^{2} \|_\infty \|\frac{u_\parallel}{\langle y \rangle}\|.
\end{align*}

The same computation is performed for all intermediate Prandtl layers, $v^i_p$, $i = 1,..,n-1$. Next, we perform the same estimate for $v^n_{pxx}$, this time using that the support is on $y \le \frac{1}{\sqrt{\eps}}$: 
\begin{align*}
\|\eps^{\frac{n}{2}} v^n_{pxx} u^0 \langle y \rangle\| \le & \|\eps^{\frac{n}{2}} v^n_{pxx}[u_\perp + \kappa u_\parallel] \langle y \rangle\| \\
\le & \eps^{\frac{n-2}{2}} \| v^n_{pxx}\eps  \langle y \rangle^{2} \|_\infty  \|u_\perp \langle y \rangle^{-1}\| + \eps^{\frac{n}{2}} \eps^{-\frac{1}{2}(\frac{5+}{2})} |\sqrt{\eps}\kappa| \\ 
&\times  \| v^n_{pxx} \eps^{\frac{1}{2}(\frac{3+}{2})} \langle y \rangle^{\frac{3+}{2}} \|_\infty \|\frac{u_\parallel}{\langle y \rangle^{\frac{3+}{2}}}\|.
\end{align*}

Next, we treat contributions from $v^i_{pyy}$ for $i = 2,...,n$, again using the rapid decay of this quantity: 
\begin{align*}
\|u^0 \sum_{i = 2}^n \sqrt{\eps}^i v^i_{pyy} \langle y \rangle\| \lesssim \eps \| v^i_{pyy} \langle y \rangle^{2} \|_\infty \Big[ \|u^0_\perp \langle y \rangle^{-1}\| + |\kappa| \|\frac{u_{\parallel}}{\langle y \rangle}\| \Big].
\end{align*}

The remaining, higher order shear terms can be treated as in the $\Delta v^1_e$ case for $u_\perp$, whereas for the parallel component:  
\begin{align*}
\| \eps \sqrt{\eps}^{i-1} \Delta v^i_e \kappa u_{\parallel} \langle y \rangle \| \lesssim \eps^{\frac{3}{4}}|\kappa|. 
\end{align*}

On the right-hand side, we majorize using $\| F \langle y \rangle \|$. This concludes the lemma. 

\end{proof}

\begin{lemma} \label{Lemma.kappa} \normalfont Let $u^0 = u_\perp + \kappa u_{\parallel}$ solve equation (\ref{sys.u0}). Assume the nondegeneracy condition, (\ref{charlie.3}). Then the following estimate is valid: 
\begin{align} \label{kappa.est}
\sqrt{\eps}| \kappa| \lesssim \eps^{\frac{1}{4}} \| u_\perp \|_{\Upsilon} + \eps^{\frac{1}{4}} \| u^0_{yy} \langle y \rangle \{1 - \chi(Y) \} \| + \| F \langle y \rangle^{\frac{1}{2}+} \|.
\end{align}
\end{lemma}
\begin{proof} Our starting point is the equation: 
\begin{align*}
\bold{d}(\text{Equation } (\ref{start})) = \bold{d}(F). 
\end{align*}

We know that $\bold{d}(\mathcal{L}_{\parallel} u_\perp) = 0$, and that, by the nondegeneracy condition, (\ref{charlie.3}),
\begin{align*}
|\sqrt{\eps} \kappa \bold{d}(S(u_{\parallel}))| \gtrsim  \sqrt{\eps} |\kappa|
\end{align*}

It remains thus to estimate $\bold{d}(\cdot)$ of the latter two terms in (\ref{start}) and $\bold{d}(F)$. We first estimate: 
\begin{align*}
|\bold{d}(F)| = |\int_0^\infty K(y) I_y[F]| \le \| K \|_1 \| I_y[F] \|_\infty \lesssim \| F \|_1 \lesssim \| F \langle y \rangle^{\frac{1}{2}+} \|. 
\end{align*}

The contributions from $\tilde{J}$ and $\sqrt{\eps} A u_\perp$ are majorized easily by: 
\begin{align*}
|\bold{d}( \tilde{J} )| + |\bold{d}(\sqrt{\eps}A u_\perp)| \lesssim& \| \tilde{J} \langle y \rangle \| + \sqrt{\eps} \| A u_\perp \langle y \rangle \| \\
\lesssim & \eps^{\frac{3}{4}} |\kappa| + \sqrt{\eps} \| u_\perp \|_{\Upsilon}. 
\end{align*}

We focus on the two main contributions, beginning with:
\begin{align*}
|\bold{d}(\bar{v}^1_e u_{\perp yy})| &\lesssim \| \bar{v}^1_e u_{\perp yy} \|_1 \le \| \bar{v}^1_e u_{\perp yy} \chi(Y) \|_1 + \| \bar{v}^1_e u_{\perp yy} \{1 - \chi(Y) \} \|_1.
\end{align*}

\noindent For the localized term, we use: 
\begin{align*}
\| \bar{v}^1_e u_{\perp yy} \chi(Y) \|_1 \le & \sqrt{\eps} \| u_{\perp yy} \langle y \rangle \|_2 \| \chi(Y) \|_2 \le  \eps^{\frac{1}{4}} \| u_\perp \|_{\Upsilon}.
\end{align*}

\noindent For the far-field term, we estimate
\begin{align*}
\| \bar{v}^1_e u_{\perp yy} \{1 - \chi(Y) \} \|_1 \le & \| u_{\perp yy} \langle y \rangle \{1- \chi(Y) \} \| \| \langle y \rangle^{-1} \{1 - \chi(Y)\} \| \\
\le & \eps^{\frac{1}{4}} \| u_{\perp yy} \langle y \rangle \{1 - \chi(Y) \} \|
\end{align*}

\noindent Above, we have used the inequality: 
\begin{align*}
\| \langle y \rangle^{-1} \{1 - \chi(\frac{y}{a}) \} \|_{L^2}^2 =& \int_a^\infty \langle  y \rangle^{-2} \ud y = a^{-1}, 
\end{align*}

\noindent with the particular choice of $a = \eps^{-\frac{1}{2}}$. 

We move to the shear $u_\perp$ term: 
\begin{align*}
\|\eps \Delta v^1_e u_\perp \|_1 &\le \|\eps \delta v^1_e \frac{u_\perp}{\langle y \rangle} \langle y \rangle \|_1  \le \sqrt{\eps} \| \Delta v^1_e \langle Y \rangle \| \| u_{\perp y} \| \\
& \le \eps^{\frac{1}{4}} \| \Delta v^1_e Y \|_{L^2_Y} \| u_{\perp y} \|  \lesssim \eps^{\frac{1}{4}} \| u_\perp \|_{\Upsilon},
\end{align*}

\noindent where we have used that $\| \cdot \| = \eps^{-\frac{1}{4}} \| \cdot \|_{L^2_Y}$, the factor of $\eps^{-\frac{1}{4}}$ arising from the Jacobian. This concludes the proof. 

\end{proof}

\begin{lemma}[Multiscale Estimate] \label{lemma.Z.scale.1} \normalfont Assume (\ref{charlie.1}) - (\ref{charlie.3}) and  (\ref{assume.Euler.i}). Then: 
\begin{align}
\begin{aligned} \label{multiscale}
\|u^0_{yy} \{ 1 &- \chi(\frac{Y}{\gamma}) \} \langle y \rangle\|^2 + \|\eps^{\frac{1}{4}}u^0_{yy} \frac{1}{\sqrt{v^1_e}} \{ 1 - \chi(\frac{Y}{\gamma}) \} \langle y \rangle\|^2 \\
& \le (\eps^\infty + \delta_s) \| u^0 \|_{B}^2 + \|F \min \{ \frac{\langle y \rangle}{v^1_e}, \frac{\eps^{-\frac{1}{4}} \langle y \rangle }{\sqrt{v^1_e}} \}  \{ 1 - \chi(\frac{Y}{\gamma}) \} \|^2.
\end{aligned}
\end{align}
\end{lemma}

\begin{proof} We compute the inner-product: 
\begin{align*}
(\mathcal{L}_\delta u^0, u^0_{yy} \frac{1}{w} |\{ 1 - \chi(\frac{Y}{\gamma}) \} |^2 \langle y \rangle^{2}) = (F, \frac{1}{w} u^0_{yy} |\{ 1 - \chi(\frac{Y}{\gamma}) \} |^2 \langle y \rangle^{2}). 
\end{align*}

On the left-hand side, we begin with:
\begin{align} \n
- &( u^0_{yyy}, \frac{1}{w} u^0_{yy} |\{ 1 - \chi(\frac{Y}{\gamma})  \}|^2 \langle y \rangle^{2}) =  ( \frac{|u^0_{yy}|^2}{2}, \eps^{\frac{1}{2}} \Big( |1 - \chi(\frac{Y}{\gamma}) |^2 \Big)' \frac{1}{w} \langle y \rangle^{2})  \\ \label{newz.1}
&- ( \frac{\sqrt{\eps}}{2} |u^0_{yy}|^2, \frac{w'}{w^2} \langle y \rangle^{2}  |\{1 - \chi(\frac{Y}{\gamma})  \}|^2) + (  \frac{1}{w} |u^0_{yy}|^2, |\{ 1 - \chi(\frac{Y}{\gamma})  \}|^2 \langle y \rangle).
\end{align}

First, we note that (\ref{newz.1}.1), (\ref{newz.1}.2), and (\ref{newz.1}.3) are favorable signed terms. First, $\chi' \le 0$ by design (see \ref{basic.cutoff}). It is also clear that $w_{y} w^{-2} \le - (v^1_e)^{-\frac{1}{2}}$. Next, upon using the positivity of $v^1_e$, we estimate 
\begin{align*}
\frac{v_s}{w} \{1 - \chi(\frac{Y}{\gamma})  \} \ge & \frac{v^1_e }{w} \{1 - \chi(\frac{Y}{\gamma})  \} - \frac{\sum_{i = 2}^n \sqrt{\eps}^{i-1}v^i_e}{w} \{1 - \chi(\frac{Y}{\gamma})  \} \\
& - \frac{\sum_{i = 0}^{n-1} \sqrt{\eps}^i v^n_p [1 - \chi(\frac{Y}{\gamma}) ]}{w} - \frac{\sqrt{\eps}^n |v^n_p|}{w} \{1 - \chi(\frac{Y}{\gamma})  \} \\
\gtrsim & 1 \{1 - \chi(\frac{Y}{\gamma})  \}. 
\end{align*}

Using this lower bound, we see 
\begin{align*}
( (v_s + \delta) u^0_{yy},  \frac{1}{w} u^0_{yy} |\{ 1 - \chi(\frac{Y}{\gamma})  \}|^2 \langle y \rangle^{2}) \gtrsim \|u^0_{yy} \{1 - \chi(\frac{Y}{\gamma})  \} \langle y \rangle\|^2. 
\end{align*}

Next, we move to: 
\begin{align} \label{bob}
( &- \Delta_\eps v_s u^0, u^0_{yy} \frac{1}{w} |1 - \chi(\frac{Y}{\gamma}) |^2 \langle y \rangle^{2}) \\ \n
& = - (  [\sum_{i = 0}^n \sqrt{\eps}^i \Delta_\eps v^i_p + \sum_{i  = 1}^n \eps \sqrt{\eps}^{i-1} \Delta v^i_e] u^0, u^0_{yy} \frac{1}{w} |[1 - \chi(\frac{Y}{\gamma}) ]|^2 \langle y \rangle^{2}).
\end{align}

First, using the rapid decay of $v^i_{pyy}$ for $i = 0,...,n$ and $v^j_{pxx}$ for $j = 0,...,n-1$, we immediately estimate: 
\begin{align*}
|( \Big[\sum_{i = 0}^n \sqrt{\eps}^i v^i_{pyy} + \sum_{j = 0}^{n-1} \sqrt{\eps}^j \eps v^j_{pxx} \Big] u^0, u^0_{yy} \frac{1}{w} |1 - \chi(\frac{Y}{\gamma}) |^2 \langle y \rangle^{2})| \lesssim \eps^{\infty} \| u^0 \|_{B}^2.
\end{align*}

This leaves the terms $\eps^{\frac{n+2}{2}} v^n_{pxx}$ and the Euler contributions. First, 
\begin{align*}
|(\ref{bob}.2)[v^n_{pxx}]| \le &\|u^0_{yy} \langle y \rangle\{1 - \chi(\frac{Y}{\gamma})  \}\| \cdot \Big[ \sqrt{\eps}^n \|u_{\perp} \langle y \rangle^{-1}\| \| Y^2 v^1_{pxx} w^{-1}\|_\infty \\
& + \eps^{\frac{n}{2}} \eps^{-\frac{3+}{4}-\frac{1}{2}} |\sqrt{\eps}\kappa| \cdot \|\frac{u_\parallel}{\langle y \rangle^{\frac{1+}{2}}}\| \| Y^{\frac{3+}{2}} v^1_{pxx} w^{-1}\|_\infty\Big].
\end{align*}

We now estimate the shear terms, with the $i > 1$ case following in a similar manner to the $i = 1$ case: 
\begin{align*}
(- \eps \Delta &v^1_e u^0, u_{\perp yy} \{1 - \chi(\frac{Y}{\gamma}) \} \frac{\langle y \rangle^2}{w}) \\
= & (- \eps \Delta v^1_e u_\perp, u_{\perp yy} \{1 - \chi(\frac{Y}{\gamma}) \} \frac{\langle y \rangle^2}{w}) +  (- \eps \Delta v^1_e \kappa u_\parallel, u_{\perp yy} \{1 - \chi(\frac{Y}{\gamma}) \} \frac{\langle y \rangle^2}{w}) \\
\lesssim & \| \Delta v^1_e  Y^2 \frac{1}{w}\| \| \frac{u_\perp}{\langle y \rangle} \| \| u_{\perp yy} \{1- \chi \} \langle y \rangle \| + |\eps^{\frac{1}{4}}\kappa| \| \Delta v^1_e Y \frac{1}{w} \|_{L^2_Y} \| u_{\perp yy} \{1 - \chi \} \langle y \rangle \| \\
\lesssim & \delta_s \| u^0 \|_B \| u_{\perp yy} \langle y \rangle \{1 - \chi \} \|. 
\end{align*}

The right-hand side is majorized immediately using Cauchy-Schwartz by $\|F \frac{\langle y \rangle}{v^1_e} \{ 1 - \chi(\frac{Y}{\gamma}) \} \| \|u^0_{yy} \{ 1 - \chi(\frac{Y}{\gamma})  \} \langle y \rangle\|$ or by $\| F \eps^{-\frac{1}{4}} \frac{\langle y \rangle}{\sqrt{v^1_e}}  \| \|u^0_{yy} \eps^{\frac{1}{4}} \frac{\langle y \rangle}{\sqrt{v_e^1}} [1 - \chi(\frac{Y}{\gamma}) ]\|$.

\end{proof}

We now close the section with a straightforward $L^\infty$ embedding:
\begin{lemma} \label{lemma.unif.u} \normalfont $|\eps^{\frac{1+}{2}}u_\perp|_\infty \lesssim \| u^0 \|_{B}$.
\end{lemma}
\begin{proof} First, using $u_\perp(0) = 0$, we write: $u_\perp = \int_0^y  u_{\perp y} \le \|u_{\perp y} \langle y \rangle^{\frac{1+}{2}}\|$. Next, using $u_{\perp y}(\infty) = 0$, we use Hardy to majorize:
\begin{align*}
|u_\perp| \lesssim& \|u_{\perp y}\|_{loc} + \|u_{\perp yy} y^{\frac{3+}{2}}\| \\
\lesssim &\| u^0 \|_{B} + \sqrt{\eps}^{1 - \frac{3+}{2}} \eps^{-\frac{1}{4}} \|\eps^{\frac{1}{4}}u_{\perp yy}y \frac{1}{\sqrt{v_e}}[1 - \chi(\frac{Y}{\gamma}) ]\| \\
& + \sqrt{\eps}^{-(0+)(1 - \frac{3+}{2})} \|u_{\perp yy} y^1 \chi(\eps^{0+}y)\| \\
\lesssim & \eps^{\frac{1 - 2+}{2}} \| u^0 \|_{B}.
\end{align*}

\end{proof}

\subsection{Computation of Degree} \label{section.euler.example}

We now compute $\bold{d}(S(u_\parallel))$. We first use the particular form of $f = f^{(1)}$ given in (\ref{defn.f1.special}) to obtain the condition: 
\begin{align}  \label{part.f}
&\int_0^{\infty} K(y) r(y) \ud y \\ \n
& =  - u_{\parallel y}|_{x = 0}(0) u^1_e|_{x = 0}(0) e^{-\int_1^0 v_{\parallel}} + \int_0^\infty K(y) f(y) \ud y \\ \n
& =  - u_{\parallel y}|_{x = 0}(0) u^1_e|_{x = 0}(0) e^{-\int_1^0 v_{\parallel}} + \int_0^\infty K(y) \{ - u^0_p u^1_{ex}(0) \\ \n
& \hspace{4 mm}- u^0_{px} u^1_e(0)  -\bar{v}^1_{eY}(0) y u^0_{py} -  u^0_{eY}(0) y u^0_{px} - v^0_p u^0_{eY}(0) \\ \n
& \hspace{4 mm} + g^{(1)}_{ext}\}|_{x =0}. 
\end{align}

\noindent Note that we have referred to the definition in (\ref{def.forcing}) and retained the lowest order terms in $f$. We now record the following identity for future use:
\begin{lemma} \normalfont 
\begin{align} 
\begin{aligned} \label{cg.pos}
 \int K(y) \Big[ r(y) &+ v^1_{eY}(0)\{ y u_{\parallel y} - u^0_p \} \Big] \ud y \\
 & =-\int K(y) \Big[ u^0_{eY}(0) y u^0_{px} + v^0_p u^0_{eY}(0)  \Big] \ud y + \int K(y) g^{(1)}_{ext} \ud y. 
 \end{aligned}
\end{align}
\end{lemma}
\begin{proof} The $v^1_{eY}(0)$ terms on the left-hand side cancel out the first and third terms in the integral in (\ref{part.f}). This leaves: 
\begin{align*}
\text{LHS of }(\ref{cg.pos})=& - u_{\parallel y}|_{x = 0}(0) u^1_e|_{x = 0}(0)e^{-\int_1^0 v_{\parallel}} + \int K(y) \{ - u^1_e|_{x = 0}(0) u^0_{px} \\ & - u^0_{eY}(0) y u^0_{px} - v^0_p u^0_{eY}(0) + g^{(1)}_{ext} \} \\
= & - u^1_e|_{x = 0}(0) \Big[ u_{\parallel y}|_{x = 0}(0) e^{-\int_1^0 v_{\parallel}} + \int K(y)u^0_{px} \Big] \\
& - \int K(y) \Big[ u^0_{eY}(0) y u^0_{px} + u^0_{eY}(0) v^0_p + g^{(1)}_{ext} \Big] \\
= & - \int K(y) \Big[- g^{(1)}_{ext} + u^0_{eY}(0) y u^0_{px} + u^0_{eY}(0) v^0_p \Big].
\end{align*}

Above, we have used the identity: 
\begin{align*}
 u_{\parallel y}|_{x = 0}(0) &+ \int u_{\parallel} u^0_{px} e^{-\int_0^y v_{\parallel}} \\
 & =  u_{\parallel y}|_{x = 0}(0) - \int v_{\parallel} u^0_{py} e^{- \int_0^y v_{\parallel}} + \int u^0_{pyy} e^{- \int_0^y v_{\parallel}} = 0,
\end{align*}

\noindent upon integrating by parts the $u^0_{pyy}$ term. 
\end{proof}

We will next make a reduction of $\bold{d}(S(u_{\parallel}))$ to the quantity $\mathfrak{n}$, which we now define: 
\begin{align}  \label{ndef}
\mathfrak{n} := & \int_0^\infty K(y) \Big[g^{(1)}_{ext} + u^0_{eY}(0) \{ y u^0_{px} + v^0_p \}  + u^0_e(0) \int_0^\infty \Delta v^1_e \ud Y\Big] \ud y.
\end{align}

\begin{lemma} \normalfont The following inequality holds:
\begin{align*}
|\bold{d}(S(u_{\parallel}))| \gtrsim |\mathfrak{n}| - \sqrt{\eps}. 
\end{align*}
\end{lemma}
\begin{proof} We begin with a rewriting of the latter two terms in $S$:
\begin{align*}
& \frac{\bar{v}^1_e}{\sqrt{\eps}} u_{\parallel yy} -  \sqrt{\eps} \Delta v^1_e u_{\parallel} \\
= &  \frac{\bar{v}^1_e}{\sqrt{\eps}} u_{\parallel yy} - \sqrt{\eps}  \Delta v^1_e u^0_p - \sqrt{\eps}  \Delta v^1_e u^0_e(0) \\
= &  \p_y \{ \frac{\bar{v}^1_e}{\sqrt{\eps}} u_{\parallel y} - v^1_{eY} u^0_p \} - \sqrt{\eps}  v^1_{exx} u^0_p - \sqrt{\eps}  u^0_e(0) \Delta v^1_e \\
= &  \p_y \{ \frac{\bar{v}^1_e}{\sqrt{\eps}} u_{\parallel y} - v^1_{eY} u^0_p \} - \sqrt{\eps}  v^1_{exx} u^0_p - u^0_e(0)  \p_y   I_Y[ \Delta v^1_e]
\end{align*}

Combining now with $A  u_{\parallel} = r'(y)$, we take $\bold{d}(\cdot)$:
\begin{align*}
\bold{d} \Big(  A u_{\parallel} &+  \frac{\bar{v}^1_e}{\sqrt{\eps}} u^0_{\parallel yy} -  \sqrt{\eps} \Delta v^1_e u^0_{\parallel} \Big) \\
= & \int_0^\infty K(y) \Big[  r(y) +  \frac{\bar{v}^1_e}{\sqrt{\eps}} u^0_{\parallel y} -   v^1_{eY} u^0_p  \Big] \ud y \\
& - \int K(y) \sqrt{\eps}  v^1_{exx} u^0_p \ud y  - \int K(y)  u^0_e(0) I_Y[\Delta v^1_e] \ud y. 
\end{align*}

We estimate the middle term immediately via: 
\begin{align*}
|\int K(y) \sqrt{\eps}  v^1_{exx} u^0_p| \lesssim \sqrt{\eps}  \| K \| \| u^0_p \| \| v^1_{exx} \|_\infty. 
\end{align*}

The lowest order terms are now collected from the first integral upon using the identity (\ref{cg.pos}):
\begin{align*}
 \int K(y) &\Big[ r(y) + \bar{v}^1_{eY}(0) u^0_{\parallel y} y - v^1_{eY}(0) u^0_p \Big] \ud y \\
= & -\int K(y) \Big[ u^0_{eY}(0) y u^0_{px} + v^0_p u^0_{eY}(0) \Big] \ud y + \int K(y) g^{u, 1}_{ext, p} \ud y. 
\end{align*} 

We have Taylor expanded: 
\begin{align*}
\bar{v}^1_e = \bar{v}^1_{eY}(0) Y + \phi(Y) Y^2, \hspace{3 mm} v^1_{eY} = v^1_{eY}(0) + Y \varphi(Y),
\end{align*}

\noindent which, upon inserting into the first integral produces the following error terms: 
\begin{align*}
|\int K(y) \Big[  \sqrt{\eps} y^2 \phi(Y) u^0_{\parallel y} -  \sqrt{\eps} \varphi(Y) y u^0_p \Big] \ud y| \lesssim \sqrt{\eps} . 
\end{align*}

The lowest order term from the third integral reads: 
\begin{align*}
- \int K(y)  & u^0_e(0) I_0[ \Delta v^1_e] \ud y  = -  u^0_e(0) \int_0^\infty K(y) \ud y \int_0^\infty \Delta v^1_e \ud Y. 
\end{align*}

A similar Taylor expansion shows that the error term can be majorized by $\sqrt{\eps} |\kappa|$. We thus arrive at the following leading order expression: 
\begin{align} \label{loe}
\int_0^\infty K(y) \Big[g^{u, 1}_{ext, p} + u^0_{eY}(0) \{ y u^0_{px} + v^0_p \}  + u^0_e(0) \int_0^\infty \Delta v^1_e \ud Y\Big] \ud y,
\end{align}

\noindent which proves the desired estimate. 

\end{proof}

\begin{corollary}
\begin{align*}
|\bold{d}(S(u_{\parallel}))| \gtrsim 1. 
\end{align*}
\end{corollary}
\begin{proof} By the previous lemma, it suffices to provide a lower bound on $\mathfrak{n}$. By invoking the assumption that all the shear terms are size $\delta_s$, we obtain 
\begin{align*}
|\mathfrak{n}| \gtrsim & |\int_0^\infty K(y)g^{u, 1}_{ext, p} \ud y| - \bigO(\delta_s) \gtrsim 1, 
\end{align*}

\noindent by our non-degeneracy assumption. 

\end{proof}

\section{Solution to DNS and NS} \label{Subsection.nonlin}

\subsection{Nonlinear \textit{a-priori} Estimate}

Define the following linear combinations: 

\begin{definition} \normalfont
\begin{align*}
&\mathcal{N}_{X_1} := (\p_x N, q_x + q_{xx} + q_{yy} + \eps^{-\frac{3}{8}}\eps^2 v_{xxxx} + \eps^{-\frac{3}{8}}\eps^{-\frac{1}{8}}\eps u_s v_{xxyy}), \\
&\mathcal{N}_{Y_w} := (\p_x N, \{ \eps q_x + \eps q_{xx} + \eps q_{yy} + \eps^2 v_{xxxx} + \eps u_s v_{xxyy}\} w^2 \\
& \hspace{10 mm} + \{\eps^2 v_{xxxx} + \eps^{-\frac{1}{8}} \eps u_s v_{xxyy} \}) \\
&\mathcal{B}_{X_1} := (F_{u^0}, q_x + q_{xx} + q_{yy} + \eps^{-\frac{3}{8}}\eps^2 v_{xxxx} + \eps^{-\frac{3}{8}}\eps^{-\frac{1}{8}}\eps u_s v_{xxyy}), \\
&\mathcal{B}_{Y_w} := (F_{u^0}, \{ \eps q_x + \eps q_{xx} + \eps q_{yy} + \eps^2 v_{xxxx} + \eps u_s v_{xxyy}\} w^2 \\
& \hspace{10 mm} + \{\eps^2 v_{xxxx} + \eps^{-\frac{1}{8}} \eps u_s v_{xxyy} \}) \\
&\mathcal{F}_{X_1} := (g_{(q)}, q_x + q_{xx} + q_{yy} + \eps^{-\frac{3}{8}}\eps^2 v_{xxxx} + \eps^{-\frac{3}{8}}\eps^{-\frac{1}{8}}\eps u_s v_{xxyy}), \\
&\mathcal{F}_{Y_w} := (g_{(q)}, \{ \eps q_x + \eps q_{xx} + \eps q_{yy} + \eps^2 v_{xxxx} + \eps u_s v_{xxyy}\} w^2 \\
& \hspace{10 mm} + \{\eps^2 v_{xxxx} + \eps^{-\frac{1}{8}} \eps u_s v_{xxyy} \})
\end{align*}
\end{definition}

\begin{lemma}[Boundary Estimates] \normalfont Let $j = 0, 1$. The following estimate holds: 
\begin{align} \label{key.z.est}
\| \p_x^j \{ F_{u^0} \} w \frac{\langle y \rangle}{u_s}\|_{x = 0} \lesssim \begin{cases} \eps^{-\frac{1}{4}} \|u^0 \|_{B} \text{ if } w = 1, \\ \eps^{-\frac{3}{4}} \| u^0 \|_{B} \text{ if } w = \frac{1}{v_e} \langle y \rangle \\
\eps^{-\frac{1}{4}} \|u^0 \|_B \text{ if } w = \frac{1}{v_e} \end{cases}. 
\end{align}
\end{lemma}
\begin{proof} First, set $j = 0$. The $j = 1$ case follows in an identical manner. First, let $w = 1$. We estimate immediately upon consulting (\ref{B.norm}): 
\begin{align*}
\|\frac{v_{sx}}{u_s} u^0_{yy} \langle y \rangle\| \le & \| \frac{v_{sx}}{u_s} \|_\infty \{ \| u_{\perp yy} \langle y \rangle \| + \eps^{-\frac{1}{4}}|\eps^{\frac{1}{4}}\kappa| \| u_{\parallel yy} \langle y \rangle \| \\
\lesssim &\eps^{-\frac{1}{4}} \|u^0 \|_B.  
\end{align*}

\noindent  Above we have used that $v_{sx}|_{y = 0} = 0$, and so $v_{sx}(0,0) = 0$. Next, we treat $|u^0 \Delta_\eps v_{sx} \frac{\langle y \rangle}{u_s}\|$. For the region $y \le 1$, we estimate easily using the boundary condition $u^0(0) = 0$: 
\begin{align*}
\|u^0 \Delta_\eps v_{sx} \frac{\langle y \rangle}{u_s} \chi\| \le \|\Delta_\eps v_{sx} \langle y \rangle \chi\|_\infty \| u^0, u^0_y \|_{loc} \lesssim \eps^{-\frac{1}{4}} \|u^0 \|_B.
\end{align*}

For the region $y \ge 1$, we may omit the weight of $u_s$ and simply estimate: 
\begin{align*}
\| u^0 \Delta_\eps v_{sx} y\| \le \| \Delta_\eps v_{sx} y^2 \|_\infty \| \frac{u^0}{y} \| \lesssim \| u^0_y \| \lesssim \eps^{-\frac{1}{4}} \| u^0 \|_B.
\end{align*}

Next, we let $w = \frac{1}{v_e} \langle y \rangle^m$ for $m = 0, 1$.  
\begin{align*}
\|v_{sx} u^0_{yy} & \langle y \rangle^{m+1} \frac{1}{v_e} [1-\chi(Y)] \| \\
& \le \eps^{-\frac{1}{4}-\frac{m}{2}} \| \frac{ v_{sx}}{\sqrt{v_e}} Y^m\|_\infty \| \frac{\eps^{\frac{1}{4}}}{\sqrt{v_e}} u^0_{yy} \langle y \rangle [1-\chi(Y)]\| \lesssim  \eps^{-\frac{1}{4}- \frac{m}{2}} \|u^0 \|_{B}.
\end{align*}

Next, 
\begin{align*}
&\|v_{sx} u_{\perp yy} \langle y \rangle^{m+1} \frac{1}{v_e} \chi(Y)\| \lesssim  \eps^{-\frac{m}{2}} \|Y^m \frac{1}{v_e} \chi(Y)\|_\infty \|u_{\perp yy} \langle y \rangle\| \lesssim  \eps^{-\frac{m}{2}} \| u^0 \|_{B}, \\
&\| v_{sx} \kappa u_{\parallel yy} \langle y \rangle^{m+1} \frac{1}{v_e} \| \lesssim \eps^{-\frac{1}{4}} |\eps^{\frac{1}{4}} \kappa| \lesssim \eps^{-\frac{1}{4}} \| u^0 \|_B.
\end{align*}

We now move to $\| u^0 \Delta_\eps v_{sx} y^2 \frac{1}{v_e}\|$. It is convenient to split: 
\begin{align*}
\Delta_\eps v_{sx} = \sum_{i = 0}^n \sqrt{\eps}^i \Delta_\eps v^i_{px} + \eps \sum_{j = 1}^n \sqrt{\eps}^{j-1} \Delta v^j_{ex}.
\end{align*}

We treat the $i = 0$ case, with the higher order Prandtl terms following similarly. First, using the rapid decay of $v^0_p$, we estimate: 
\begin{align*}
&\| u_\perp \Delta_\eps v^0_{px} \frac{\langle y \rangle^{m+1}}{v^1_e} \| \lesssim \| \frac{u_\perp}{y} \| \lesssim \| u^0 \|_B \\
&\| \kappa u_\parallel \Delta_\eps v^0_{px} \frac{\langle y \rangle^{m+1}}{v^1_e} \| \lesssim \eps^{-\frac{1}{4}} \| u^0 \|_B.
\end{align*}

Next, we move to the Euler contributions. We treat the $j = 1$ case, with the higher order Euler terms following similarly:
\begin{align*}
&\| u_\perp \eps \Delta v^1_{ex} \frac{\langle y \rangle^{m+1}}{v^1_e|_{x = 0}} \| \lesssim \eps \eps^{-\frac{1}{2}(m+2)} \|\frac{u_\perp}{\langle y \rangle} \|_2 \| \Delta v^1_{ex} Y^{m+2} \frac{1}{v^1_e|_{x = 0}} \|_\infty, \\
& \hspace{29 mm} \lesssim \eps^{-\frac{m}{2}} \| u^0 \|_B. \\
&\| \kappa u_{\parallel} \eps \Delta v^1_{ex} \frac{\langle y \rangle^{m+1}}{v^1_e|_{x = 0}} \| \lesssim \eps \eps^{-\frac{1}{2}(m+1)} \eps^{-\frac{1}{4}} |\eps^{\frac{1}{4}}\kappa| \eps^{-\frac{1}{4}} \| \Delta v^1_{ex} Y^{m+1} \frac{1}{v^1_e|_{x = 0}}\|_{L^2_Y} \\
& \hspace{29 mm} \lesssim \eps^{-\frac{m}{2}} \| u^0 \|_B. 
\end{align*}

This concludes the proof.

\end{proof}

\begin{lemma}
\begin{align}
&|\mathcal{B}_{X_1}| \lesssim \eps^{-\frac{1}{2}} \| u^0 \|_B^2 + o(1) \| v \|_{X_1}^2\\
&|\mathcal{B}_{Y_{w_0}}| \lesssim L \eps^{-\frac{1}{2}} \| u^0 \|_B^2 + o(1) \| v \|_{Y_{w_0}}^2 \\
&|\mathcal{B}_{Y_{w_1}}| \lesssim \eps^{\frac{3}{8}} \| u^0 \|_B^2 + o(1) \| v \|_{Y_{w_1}}^2. 
\end{align}
\end{lemma}
\begin{proof} First, fix $w = 1$. We compute by integrating by parts first in $y$ and second in $x$ for the first term below: 
\begin{align} \n
(F_{u^0}, q_{xx}) = &(\p_y \{ v_{sx}u^0_y - v_{sxy} u^0 \}, q_{xx}) - (\eps v_{sxxx} u^0, q_{xx}) \\ \n
= & -(v_{sx}u^0_y - v_{sxy} u^0, q_{xxy})   - (\eps v_{sxxx} u^0, q_{xx}) \\ \n
= & (v_{sxx}u^0_{y} - v_{sxxy}u^0, q_{xy})  - (v_{sx}u^0_{y} - v_{sxy}u^0, q_{xy})_{x = L} \\ \label{treat.bxx}
& +  (v_{sx}u^0_{y} - v_{sxy}u^0, q_{xy})_{x = 0} - (\eps v_{sxxx} u^0, q_{xx}).
\end{align}

\noindent  We first estimate (\ref{treat.bxx}.3) and  (\ref{treat.bxx}.4):
\begin{align*}
&|(\ref{treat.bxx}.3)| \lesssim  \Big[ \| \frac{v_{sx}}{u_s}\|_\infty \| u^0_y\| + \|v_{sxy} \langle y \rangle \| \|\frac{u^0}{u_s \langle y \rangle} \| \Big] \|u_s q_{xy} \|_{x = 0} \\
& \hspace{12 mm} \lesssim \eps^{-\frac{1}{4}} \|u^0 \|_B |||q|||_1, \\
&|(\ref{treat.bxx}.4)| \le   \|v_{sxxx} \sqrt{\eps} y\|_\infty \|\frac{u^0}{\langle y \rangle}\| \| \sqrt{\eps} q_{xx} \| \lesssim \eps^{-\frac{1}{4}} \| u^0 \|_B |||q|||_1.
\end{align*}

\noindent  Above, we have used that $v_{sx}|_{y = 0} = 0$ and that: 
\begin{align*}
&\| v_{sxy} \langle y \rangle \|_\infty \le \sum_{i = 0}^n \sqrt{\eps}^i \| v^i_{pxy} \langle y \rangle \|_\infty +  \sum_{i = 1}^n \sqrt{\eps}^{i-1}\| v^i_{exy} \sqrt{\eps} y \|_\infty \lesssim 1, \\
&\|v_{sxxx} \sqrt{\eps}y \|_\infty \lesssim  \sqrt{\eps} \sum_{i = 0}^n \sqrt{\eps}^i \|v^i_{pxx} \langle y \rangle\|_\infty + \sum_{i = 1}^n \sqrt{\eps}^{i-1} \| v^i_{exxx} \sqrt{\eps}y \|_\infty. 
\end{align*} 

For the $w \neq 1$ case, we invoke the bottom estimate in (\ref{key.z.est}):
\begin{align*}
|(v_{sx}u^0_{yy} - u^0 \Delta_\eps v_{sx}, q_{xx} w^2)| \le& \eps^{-\frac{1}{2}} \sqrt{L} \| \{v_{sx} u^0_{yy} - u^0 \Delta_\eps v_{sx} \} w \| \| \sqrt{\eps} q_{xx} w \| \\
\lesssim &\eps^{-\frac{3}{4}} \sqrt{L} \| u^0 \|_B |||q|||_w. 
\end{align*}

\begin{align*}
|(F_{u^0}, q_{yy} w^2)| \lesssim & \sqrt{L} |(v_{sx}u^0_{yy} - u^0 \Delta_\eps v_{sx}) w| \| q_{yy}w \| \\
\lesssim & \sqrt{L} |||q|||_w \eps^{-\frac{1}{4}} \| u^0 \|_B.
\end{align*}

\begin{align}
|(F_{u^0}, q_x w^2)| \lesssim & \sqrt{L} |(v_{sx}u^0_{yy} - u^0 \Delta_\eps v_{sx}) w y|  \| \frac{ q_x}{y} w \| \\ \n
\lesssim & \sqrt{L} |||q|||_w  \begin{cases} \eps^{-\frac{1}{4}} \| u^0 \|_{B} \text{ if } w = 1, \\  \eps^{-\frac{3}{4}} \| u^0 \|_{B} \text{ if } w = \frac{1}{v_e} \langle y \rangle \end{cases}.
\end{align}

\begin{align*}
(F_{u^0}, \eps v_{xxyy} u_s w^2) = & - (\{v_{sx}u^0_{yy} - u^0 \Delta_\eps v_{sx}\}, \eps v_{xyy} u_s w^2 )_{x = 0} \\
& - (\{ v_{sxx}u^0_{yy} - u^0 \Delta_\eps v_{sxx} \}, \eps v_{xyy} u_s w^2) \\
& - (\{ v_{sx}u^0_{yy} - u^0 \Delta_\eps v_{sx} \}, \eps v_{xyy} u_{sx}w^2) \\
\lesssim & \sqrt{\eps}\|v_{sx}u^0_{yy} - u^0 \Delta_\eps v_{sx} w\| \|\sqrt{\eps} u_s v_{xyy}w\|_{x = 0} \\
& + \sqrt{L}  \sqrt{\eps}\|\{v_{sxx}u^0_{yy} - u^0 \Delta_\eps v_{sxx} \}w\| \|\sqrt{\eps} u_s v_{xyy} w\| \\
& + \sqrt{L} \sqrt{\eps} \|\{ v_{sx}u^0_{yy} - u^0 \Delta_\eps v_{sx} \} w\| \| \sqrt{\eps} u_s v_{xyy} w \| \\
\lesssim & \eps^{\frac{1}{4}} \| u^0 \|_B |||q|||_{\sqrt{\eps}w}.
\end{align*}

\begin{align*}
(F_{u^0}, \eps v_{xxyy} u_s w^2) = &-(\eps^2 [v_{sx}u^0_{yy} + u^0 \Delta_\eps v_{sx}] v_{xxx}w^2)_{x = 0} \\
& - ( \eps^2 v_{sxx} u^0_{yy}, v_{xxx} w^2 ) +  ( \eps^2 \Delta_\eps v_{sxx} u^0, v_{xxx} w^2 ) \\
\lesssim & \sqrt{\eps} \|u_s \eps^{\frac{3}{2}}v_{xxx} w\|_{x = 0} \|(v_{sx}u^0_{yy} - u^0 \Delta_\eps v_{sx}) \frac{w}{u_s}\| \\
& + \sqrt{\eps} \sqrt{L} \|v_{sxx} u^0_{yy} w - u^0 \Delta_\eps v_{sxx} w \| \| \eps v_{xxx} \sqrt{\eps}w \| \\
\lesssim & \eps^{\frac{1}{4}} \| u^0 \|_B |||q |||_{\sqrt{\eps}w}.
\end{align*}

\end{proof}

\begin{proposition} \label{propn.loop} \normalfont 
\begin{align}
\begin{aligned} \label{olddom}
&\| u^0 \|_{B}^2 \lesssim \eps^{\frac{1}{2}-} \| v \|_{Y_{1}}^2 + \eps^{\frac{1}{2}+} \| v \|_{Y_{w_0}}^2 + \eps^{\frac{1}{2} + \frac{3}{16}-}\| v \|_{X_1}^2 + \| g_{(u)} w_0 \|^2 \\
& \hspace{20 mm} + \eps^{-\frac{1}{2}} \| g_{(u)} \langle y \rangle^{\frac{1}{2}+}\|^2, \\
&\| v \|_{X_1}^2 \lesssim \mathcal{B}_{X_1} + \mathcal{F}_{X_1} + \mathcal{N}_{X_1} \\
&\| v \|_{Y_{1}}^2 \lesssim \mathcal{B}_{Y_w} + \mathcal{F}_{Y_1} + \mathcal{N}_{Y_1} \\
&\| v \|_{Y_{w_0}}^2 \lesssim o_L(1) \| v \|_{X_1}^2 + \mathcal{B}_{Y_{w_0}}  + \mathcal{F}_{Y_{w_0}} + \mathcal{N}_{Y_{w_0}}.
\end{aligned}
\end{align}
\end{proposition}
\begin{proof} First bring together estimates (\ref{estimate.J}), (\ref{kappa.est}), and (\ref{multiscale}) to obtain 
\begin{align*}
\| u^0 \|_B^2 \lesssim \|F \langle y \rangle \|^2 + \eps^{-\frac{1}{2}} \| F \langle y \rangle^{\frac{1}{2}+}\|^2. 
\end{align*} 

Next, consider $F = F_{(v)} + g_{(u)}$ as specified in (\ref{sys.u0}). The $g_{(u)}$ quantities appear in the desired estimate, so we do not treat them further. We interpolate the $F_{(v)}$ terms in the following manner: 
\begin{align} \n
\eps^{-\frac{1}{4}}\| F_{(v)} \langle y \rangle^{\frac{1}{2}+} \| \le& \eps^{-\frac{1}{4}} \| F_{(v)} \langle y \rangle \|^{\frac{1}{2}+} \| F_{(v)} \|^{\frac{1}{2}-} \\ \n
\lesssim & \eps^{-\frac{(1-)}{4}}  \| F_{(v)} \langle y \rangle \| + \eps^{-\frac{(1+)}{4}} \| F_{(v)} \| \\ \label{det.a}
\lesssim &  \eps^{-\frac{(1-)}{4}} \| F_{(v)} w_0\| +\eps^{-\frac{(1+)}{4}} \| F_{(v)} \|.
\end{align}

We now estimate each term in $F_{(v)}$ which we write here from (\ref{sys.u0}) for convenience: 
\begin{align} \label{Fv}
F_{(v)} := -2\eps u_s u_{sx}q_x|_{x = 0} - 2\eps v_{xyy}|_{x =0} - \eps^2 v_{xxx}|_{x = 0} + \eps v_s v_{xy}|_{x = 0}.
\end{align}

Starting with the higher order terms,
\begin{align*}
\|\eps^2 v_{xxx} w \|_{x = 0} \le& \| \eps^2 v_{xxx} w\{1- \chi\} \|_{x = 0} +  \| \eps^2 v_{xxx} w  \chi  \|_{x = 0} \\
\le & \sqrt{\eps} \| \eps^{\frac{3}{2}}u_s v_{xxx} w \|_{x = 0} + \eps^{\frac{1}{2}} \| \eps v_{xxx} \|^{\frac{1}{2}} \| \eps^2 v_{xxxx} \|^{\frac{1}{2}} \\
\lesssim & \sqrt{\eps} \| v \|_{Y_w} + \eps^{\frac{1}{2}+\frac{3}{32}} \| v \|_{X_1}.
\end{align*}

\noindent The identical argument is performed for (\ref{Fv}.2). 

For the fourth term, we expand $v_{xy}|_{x = 0} = u_s q_{xy}|_{x = 0} + u_{sy}q_x|_{x = 0}$, perform a Hardy type inequality for the $q_x$ term to obtain
\begin{align*}
\| \eps v_s v_{xy} w \|_{x = 0} \le& \| \eps v_s u_s q_{xy} w \|_{x = 0} + \| \eps v_s u_{sy} q_x w \|_{x = 0} \\
\le & \sqrt{\eps} \| v \|_{Y_w} + \eps \| v_s u_{sy} q_x \chi \|_{x = 0} \\
\le & \sqrt{\eps} \| v \|_{Y_w} + \eps^{\frac{3}{4}} \| \eps^{\frac{1}{4}}\frac{q_x}{\langle y \rangle}\|_{x = 0} \\
\le & \sqrt{\eps} \| v \|_{Y_w} + \eps^{\frac{3}{4}} \| v \|_{X_1}.
\end{align*}

To estimate the first term from (\ref{Fv}), we split into Euler and Prandtl: 
\begin{align*}
\| \eps u_s u_{sx}q_x w \|_{x = 0} \le& \| \eps u_s u^P_{sx} q_x w \|_{x = 0} + \eps^{\frac{3}{2}} \| u_s u^E_{sx} q_x w \|_{x = 0} \\
\le & \| u^P_{sx} w \langle y \rangle \|_\infty \eps \| \frac{q_x}{\langle y \rangle} \|_{x = 0} + \sqrt{\eps} [\|\sqrt{\eps}q_x \sqrt{\eps}w \| + \| \sqrt{\eps} q_{xx} \sqrt{\eps}w \|] \\
\lesssim & \eps^{\frac{3}{4}} \| v \|_{X_1} + \sqrt{\eps} \| v \|_{Y_w}.
\end{align*}

We have thus established: 
\begin{align*}
\| F_{(v)} w \| \lesssim \sqrt{\eps} \| v \|_{Y_w} + \eps^{\frac{1}{2}+\frac{3}{32}} \| v \|_{X_1}. 
\end{align*}

Inserting this inequality into (\ref{det.a}) with $w = 1$ and $w = w_0$ respectively gives the desired bound for $\| u^0 \|_B^2$. This concludes the first estimate of (\ref{olddom}).

We now move to the $X_1$ estimate in (\ref{olddom}). The following two bounds hold: 
\begin{align*}
&||||v||||_1^2 \lesssim \eps^{\frac{3}{8}}|||q|||^2  + (F_{u^0} + g_{(q)} +\p_x N, \eps^2 v_{xxxx} + \eps^{-\frac{1}{8}} \eps u_s v_{xxyy} )\\
&|||q|||_1^2 \lesssim o_L(1) ||||v||||_1^2 + (F_{u^0} + g_{(q)} + \p_x N, q_{xx} + q_{yy} + q_x). 
\end{align*}

\noindent We multiply the top equation by $\eps^{-\frac{3}{8}}$ and add it to the bottom equation, which establishes the desired estimate for $X_1$. To establish the $Y_1$ estimate, we multiply the bottom equation by $\eps$ and add it to the top equation. 

We next write the inequalities for general $w$: 
\begin{align*}
&||||v||||_w^2 \lesssim \eps^{\frac{3}{8}} ||| q |||_1^2 + |||q|||_{\sqrt{\eps}w}^2 + |||q|||_{\sqrt{\eps}w} |||q|||_{w_y} \\
& \hspace{15 mm}+ (F_{u^0} + g_{(q)} + \p_x N, \{ \eps^2 v_{xxxx} + \eps u_s v_{xxyy} \} w^2 + \{ \eps^2 v_{xxxx} + \eps^{-\frac{1}{8}} \eps u_s v_{xxyy} \}), \\
&|||q|||_w^2 \lesssim o_L(1) ||||v||||_w^2 + o_L(1) \| q_{xx} \|_{w_y}^2 + L \eps^{-\frac{3}{2}} \| u^0 \|_B^2 + \eps^{-1} \| g_{(q)}w \|^2 \\
& \hspace{15 mm} + (F_{u^0} + g_{(q)} + \p_x N, \{ q_{xx} + q_x + q_{yy} \} w^2).
\end{align*}

We estimate the $Y_{w_0}$ norm in the following way. Turn first to the third order equation from above. First, note that we may use the inequality $|w_y| \lesssim \sqrt{\eps}|w| + 1$ to estimate 
\begin{align*}
\| q_{xx} \|_{w_y}^2 \lesssim \| q_{xx} \|_{\sqrt{\eps}w}^2 + ||q_{xx}||^2.
\end{align*}

According to our definition of $o_L(1)$, there exists some $\sigma_1$ such that the first two majorizing terms in the $|||q|||_w$ estimate can be written as: 
\begin{align*}
L^{\sigma_1} \{ ||||q||||_{w}^2 + \| q_{xx} \|_{w_y}^2 \} \le L^{\sigma_1} ||||q||||_{w}^2 + L^{\sigma_1} \| q_{xx} \|_{\sqrt{\eps}w}^2 + L^{\sigma_1} ||q_{xx}||^2. 
\end{align*}

We now turn to the $||||q||||_w$ estimate above, and split the product: 
\begin{align*}
|||q|||_{\sqrt{\eps}w} |||q|||_{w_y} \lesssim& L^{-\sigma_2} |||q|||_{\sqrt{\eps}w}^2 + L^{\sigma_2} |||q|||_{w_y}^2 \\
\lesssim &L^{-\sigma_2} |||q|||_{\sqrt{\eps}w}^2 + L^{\sigma_2} |||q|||_{\sqrt{\eps} w}^2 + L^{\sigma_2} |||q|||_{1}^2.
\end{align*}

We select $0 < \sigma_2 << \sigma_1$. Then we multiply the $|||q|||_w^2$ equation by $\eps L^{-\sigma_2}$ and add it to the $||||q||||_w^2$ estimate. This then concludes the proof. 

\end{proof}

The above scheme closes to yield: 
\begin{align}
\begin{aligned} \label{hindi.1}
\| \bold{u} \|_{\mathcal{X}}^2 \lesssim & \eps^{\frac{1}{2}}\mathcal{N}_{X_1}  + \eps^{\frac{1}{2}-} \mathcal{N}_{Y_{w_0}} + \eps^{\frac{1}{2}+} \mathcal{N}_{Y_1} + \mathcal{F}, 
\end{aligned}
\end{align}

\noindent where 

\begin{align} \label{defn.forcing.f}
\mathcal{F} := &  \eps^{\frac{1}{2}}\mathcal{F}_{X_1}  + \eps^{\frac{1}{2}-} \mathcal{F}_{Y_{w_0}} + \eps^{\frac{1}{2}+} \mathcal{F}_{Y_1} \\ \n
& + \| g_{(u)} w_0 \|^2 + \eps^{-\frac{1}{2}} \| g_{(u)} \langle y \rangle^{\frac{1}{2}+}\|^2.
\end{align}

\begin{lemma} \label{lemma.nonlinear}
\begin{align} \label{showin}
&|\mathcal{N}_{X_1}| \lesssim \eps^{p-(1+)} \| \bar{u}^0 \|_B \| \bar{v} \|_{X_1} \| v \|_{X_1} + \eps^{p-\frac{3}{4}} \| \bar{v} \|_{X_1} \| \bar{v} \|_{X_1} \| v \|_{X_1}, \\ \label{showin2}
&|\mathcal{N}_{Y_w}| \lesssim \eps^{p-(1+)} \| \bar{u}^0 \|_B \| \bar{v} \|_{Y_w} \| v \|_{Y_w} + \eps^{p-\frac{3}{4}} \| \bar{v} \|_{X_1} \|\bar{v} \|_{Y_w} \| v \|_{Y_w} 
\end{align}
\end{lemma}
\begin{proof}  We begin with the immediate estimates: 
\begin{align*}
|\mathcal{N}_{X_1}| \lesssim  \eps^{-\frac{1}{2}} \| \p_x N \| \| v \|_{X_1} \hspace{3 mm} |\mathcal{N}_{Y_{w}}| \lesssim  \| \p_x N w \| \| v \|_{Y_w}.
\end{align*}

We now establish the following bound:
\begin{align*}
\eps^{-p}\| \p_x N \cdot w \| \lesssim \{\eps^{-\frac{1+}{2}} \|  \bar{u}^0 \|_B + \eps^{-\frac{1}{4}} \|  \bar{v} \|_{X_1} \}||| \bar{q} |||_w. 
\end{align*}

To establish this, we go term by term: 
\begin{align*}
&\| \bar{v}_y \Delta_\eps \bar{v} w \| \le \| \bar{v}_y \|_\infty \| \Delta_\eps \bar{v} w \| \\
&\|I_x[\bar{v}_y] \Delta_\eps \bar{v}_x w \| \le \| \bar{v}_y \|_\infty \| \Delta_\eps \bar{v}_x w \| \\
&\| \bar{v}_x I_x[\bar{v}_{yyy}] w \| \le \eps^{-\frac{1}{4}} \| \eps^{\frac{1}{4}} \bar{v}_x \|_{L^2_x L^\infty_y} \| \bar{v}_{yyy} w \| \\
&\| \bar{v} \bar{v}_{yyy} w \| \le \eps^{-\frac{1}{4}} \| \eps^{\frac{1}{4}} \bar{v} \|_{L^2_x L^\infty_y} \| \bar{v}_{yyy} w \| \\
&\| \eps \bar{v}_x \bar{v}_{xy}w \| \le \sqrt{\eps} \| \sqrt{\eps} \bar{v}_x \|_\infty \| \bar{v}_{xy} w \| \\
&\| \bar{v} \Delta_\eps \bar{v}_y w \| \le \eps^{-\frac{1}{4}} \| \eps^{\frac{1}{4}} \bar{v} \|_\infty \| \Delta_\eps \bar{v}_y w \| \\
&\| \bar{u}^0 \Delta_\eps \bar{v}_x w \| \le | \bar{u}^0 |_\infty \| \Delta_\eps \bar{v}_x w \| \lesssim \eps^{-\frac{(1+)}{2}} \| \bar{u}^0 \|_B |||\bar{q}|||_w, \\
&\| \bar{u}^0_{yy} \bar{v}_x w \| \le \| \bar{u}^0_{yy} \langle y \rangle \|_{L^\infty_x L^2_y} \| \bar{v}_x \langle y \rangle^{-1} w \|_{L^2_x L^\infty_y} \lesssim \| \bar{u}^0 \|_{B} |||\bar{q}|||_w.
\end{align*}

Above, we have used the following interpolation: 
\begin{align*}
\| \bar{v}_x \langle y \rangle^{-\frac{1}{2}} w \|_{L^2_x L^\infty_y} \le & \| \bar{v}_x \langle y \rangle^{-1} w \|^{\frac{1}{2}} \| \bar{v}_{xy} w\|^{\frac{1}{2}},
\end{align*}

\noindent  and Hardy's inequality. The result follows upon remarking the following basic fact. For any function $g(x,y)$ such that $g_{x = 0 \text{ OR } x= L} = 0$ and $g|_{y = \infty} = 0$: $|g|^2 \le \| g_x \| \| g_y \| + \| g \| \| g_{xy} \|$. This immediately gives: $\| \eps^{\frac{1}{4}} v \|_\infty + \| v \langle y \rangle^{-1/2} \|_\infty + \|\nabla_\eps v \|_\infty \lesssim |||q|||_1$. A basic interpolation also gives $\| \eps^{\frac{1}{4}} \bar{v}_x \|_{L^\infty_y} \le \| \sqrt{\eps} v_x \|_{L^2_y}^{\frac{1}{2}} \|v_{xy} \|_{L^2_y}^{\frac{1}{2}}$.

\end{proof}

We now select $N_0 = 1+$ and $n = 4$ to obtain  
\begin{corollary} \normalfont  Let $u^0$ solve (\ref{sys.u0}) and $v$ solve (\ref{eqn.dif.1.app}). Then the following estimate holds: 
\begin{align*}
\| \bold{u} \|_{\mathcal{X}}^2 \lesssim \| \bold{\bar{u}} \|_{\mathcal{X}}^4 + o(1) \| \bold{\bar{u}} \|_{\mathcal{X}}^2 + o(1).
\end{align*}
\end{corollary}


\newpage

\appendix

\part*{Appendix}

\section{Derivation of Equations} \label{appendix.derive}

We will assume the expansions: 
\begin{align}
&U^\eps = \tilde{u}^n_s + \eps^{N_0} u, \hspace{3 mm} V^\eps = \tilde{v}^n_s + \eps^{N_0} v, \hspace{3 mm} P^\eps = \tilde{P}^n_s + \eps^{N_0} P.
\end{align}

\noindent We will denote the partial expansions: 
\begin{align}
&u_s^i = \sum_{j = 0}^i \sqrt{\eps}^j u^j_e + \sum_{j = 0}^{i-1} \sqrt{\eps}^j u^j_p, \hspace{5 mm} \tilde{u}_s^i = u_s^i + \sqrt{\eps}^i u^i_p, \\
&v_s^i = \sum_{j = 1}^i \sqrt{\eps}^{j-1} v^j_e + \sum_{j = 0}^{i-1} \sqrt{\eps}^j v^j_p, \hspace{5 mm} \tilde{v}_s^i = v_s^i + \sqrt{\eps}^i v^i_p, \\
&P^i_s = \sum_{j = 0}^i \sqrt{\eps}^j P^j_e, \hspace{5 mm}  \tilde{P}_s^i = P_s^i + \sqrt{\eps}^i \Big\{ P^i_p + \sqrt{\eps} P^{i,a}_p \Big\}.
\end{align}

\noindent  We will also define $u^{E,i}_s = \sum_{j = 0}^i \sqrt{\eps}^j u^j_e$ to be the ``Euler" components of the partial sum. Similar notation will be used for $u^{P,i}_s, v^{E,i}_s, v^{P,i}_s$. The following will also be convenient: 
\begin{align}
\begin{aligned} \label{profile.splitting}
&u_s^E := \sum_{i =0}^n \sqrt{\eps}^i u^i_e, \hspace{3 mm} v_s^E := \sum_{i = 1}^n \sqrt{\eps}^{i-1} v^i_e, \\
&u_s^P := \sum_{i = 0}^n \sqrt{\eps}^i u^i_p, \hspace{3 mm} v_s^P := \sum_{i = 0}^n \sqrt{\eps}^i v^i_p, \\
&u_s = u_s^P + u_s^E, \hspace{3 mm} v_s = v_s^P + v_s^E. 
\end{aligned}
\end{align} 

\noindent  The $P^{i,a}_p$ terms are ``auxiliary Pressures" in the same sense as those introduced in \cite{GN} and \cite{Iyer} and are for convenience. We will also introduce the notation: 
\begin{align} \label{bar.defs}
\bar{u}^i_p := u^i_p - u^i_p|_{y = 0}, \hspace{5 mm} \bar{v}^i_p := v^i_p - v^i_p(x,0), \hspace{5 mm} \bar{v}^i_e = v^i_e - v^i_e|_{Y = 0}.
\end{align}

\subsection{$i = 0$}

We first record the properties of the leading order $(i = 0)$ layers. For the outer Euler flow, we will take a shear flow, $[u^0_e(Y), 0, 0]$. The derivatives of $u^0_e$ decay rapidly in $Y$ and that is bounded below, $|u^0_e| \gtrsim 1$. 

For the leading order Prandtl boundary layer, the equations are: 
\begin{align}
\left.
\begin{aligned} \label{Pr.leading}
&\bar{u}^0_p u^0_{px} + \bar{v}^0_p u^0_{py} - u^0_{pyy} + P^0_{px} = 0, \\
&u^0_{px} + v^0_{py} = 0, \hspace{3 mm} P^0_{py} = 0, \hspace{3 mm} u^0_p|_{x = 0} = U^0_P, \hspace{3 mm} u^0_p|_{y = 0} = - u^0_e|_{Y = 0}.
\end{aligned}
\right\}
\end{align}

It is convenient to state results in terms of the quantity $\bar{u}^0_p$, whose initial data is simply $\bar{U}^0_P := u^0_e(0) + U^0_P$. Our starting point is the following result of Oleinik in \cite{Oleinik}, P. 21, Theorem 2.1.1:
\begin{theorem}[Oleinik]   \label{thm.Oleinik} Assume boundary data is prescribed satisfying $U^0_P \in C^\infty$ and exponentially decaying $|\p_y^j \{\bar{U}^0_P - u^0_e(0)\}|$ for $j \ge 0$ satisfying: 
\begin{align} 
\begin{aligned} \label{OL.1}
& \bar{U}^0_P > 0 \text{ for } y > 0, \hspace{3 mm} \p_y \bar{U}^0_P(0) > 0, \hspace{3 mm} \p_y^2 \bar{U}^0_P \sim y^2 \text{ near } y = 0
\end{aligned}
\end{align}

\noindent  Then for some $L > 0$, there exists a solution, $[\bar{u}^0_p, \bar{v}^0_p]$ to (\ref{Pr.leading}) satisfying, for some $y_0, m_0 > 0$, 
\begin{align} \label{coe.2}
&\sup_{x \in (0,L)} \sup_{y \in (0, y_0)} |\bar{u}^0_p, \bar{v}^0_p, \p_y \bar{u}^0_p, \p_{yy}\bar{u}^0_p, \p_x \bar{u}^0_p| \lesssim 1, \\ \label{coe.1}
&\sup_{x \in (0,L)} \sup_{y \in (0, y_0)} \p_y \bar{u}^0_p > m_0 > 0. 
\end{align}
\end{theorem}

\noindent  By evaluating the system (\ref{Pr.leading}) and $\partial_y$ of (\ref{Pr.leading}) at $\{y = 0\}$ we conclude: 
\begin{align*}
\bar{u}^0_{pyy}|_{y = 0} = \bar{u}^0_{pyyy}|_{y = 0} = 0. 
\end{align*}

\subsection{$1 \le i \le n-1$}

We now list the equations to be satisfied by the $i$'th layers, starting with the $i$'th Euler layer:
\begin{align} \label{des.eul.1}
\left.
\begin{aligned}
&u^0_e \p_x u^i_e + \p_Y u^0_e v^i_e + \p_x P^i_e =: f^i_{E,1}, \\
&u^0_e \p_x v^i_e + \p_Y P^i_e  =: f^i_{E,2}, \\
&\p_x u^i_e + \p_Y v^i_e = 0, \\
&v^i_e|_{Y = 0} = - v^0_p|_{y = 0}, \hspace{5 mm} v^i_e|_{x = 0, L} = V_{E, \{0, L\}}^i \hspace{5 mm} u^i_e|_{x = 0} = U^i_{E}.
\end{aligned}
\right\}
\end{align}

For the $i$'th Prandtl layer:
\begin{align} \label{des.pr.1}
\left.
\begin{aligned}
&\bar{u} \p_x u^i_p + u^i_p \p_x \bar{u} + \p_y \bar{u} [v^i_p - v^i_p|_{y = 0}] + \bar{v} \p_y u^i_p + \p_x P^i_p - \p_{yy} u^i_p := f^{(i)}, \\  
& \p_x u^i_p + \p_y v^i_p = 0,  \hspace{5 mm} \p_y P^i_p = 0\\  
& u^i_p|_{y = 0} = -u^i_e|_{y = 0}, \hspace{5 mm} [u^i_p, v^i_p]_{y \rightarrow \infty} = 0, \hspace{5 mm} v^i_p|_{x = 0} = \text{prescribed initial data}. 
\end{aligned}
\right\}
\end{align}

The relevant definitions of the above forcing terms are given below. Note that as a matter of convention, summations that end with a negative number are empty sums.
\begin{definition}[Forcing Terms] \label{def.forcing}  
\begin{align*}
&-f^i_{E,1} := u^{i-1}_{ex} \sum_{j = 1}^{i-2} \sqrt{\eps}^{j-1} \{u^j_e + u^j_p(x,\infty) + u^{i-1}_e \sum_{j = 1}^{i-2} \sqrt{\eps}^{j-1} u^j_{ex} \\
& \hspace{15 mm} + \sqrt{\eps}^{i-2}[ \{u^{i-1}_{e} + u^{i-1}_p(x,\infty) \} u^{i-1}_{ex} + v^{i-1}_e u^{i-1}_{eY}] \\
& \hspace{15 mm} + u^{i-1}_{eY} \sum_{j = 1}^{i-2} \sqrt{\eps}^{j-1} v^j_e + v^{i-1}_e \sum_{j = 1}^{i-2} \sqrt{\eps}^{j-1} u^j_{eY} - \sqrt{\eps} \Delta u^{i-1}_e - g^{u,i}_{ext, e} \\
&-f^i_{E,2} := v^{i-1}_{eY} \sum_{j = 1}^{i-2} \sqrt{\eps}^{j-1} v^j_e + v_e^{i-1} \sum_{j = 1}^{i-2} \sqrt{\eps}^{j-1} v^j_{eY} + \sqrt{\eps}^{i-2}[v^{i-1}_e v^{i-1}_{eY} + u^{i-1}_e v^{i-1}_{ex}] \\
& \hspace{15 mm} + \{u_e^{i-1} + u^{i-1}_p(x,\infty)\} \sum_{j =1}^{i-2} \sqrt{\eps}^{j-1} v^j_{ex} + v^{i-1}_{ex} \sum_{j = 1}^{i-2} \sqrt{\eps}^{j-1} \{u^j_e + u^j_p(x,\infty) \}\\
& \hspace{15 mm} - \sqrt{\eps} \Delta v^{i-1}_e - g^{v, i}_{ext, e}, \\
&-f^{(i)} := \sqrt{\eps} u^{i-1}_{pxx} + \eps^{-\frac{1}{2}} \{ v^i_e - v^i_e(x,0) \} u^0_{py} + \eps^{-\frac{1}{2}} \{ u^0_e - u^0_e(0) \} u^{i-1}_{px} + \eps^{-\frac{1}{2}} \{ u^{P, i-1}_{sx} \\
& \hspace{15 mm} - \bar{u}^0_{sx} \} u^{i-1}_p +  \eps^{-\frac{1}{2}} \{ u^{E, i-1}_{sx}  - \bar{u}^0_{sx} \} \{u^{i-1}_p - u^{i-1}_p(x,\infty) \} + \eps^{-\frac{1}{2}} v^{i-1}_p \{ \bar{u}^{i-1}_{sy} \\
& \hspace{15 mm} - u^0_{py} \} + u^{i-1}_{px} \sum_{j =1 }^{i-1} \sqrt{\eps}^{j-1}(u^j_e + u^j_p)  + \eps^{-\frac{1}{2}} (v_s^{i-1} - v_s^1) u^{i-1}_{py} + \eps^{-\frac{1}{2}} (v^1_e \\
& \hspace{15 mm} - v^1_e(x,0)) u^{i-1}_{py} + \sqrt{\eps} u^i_{eY} \sum_{j = 0}^{i-1} \sqrt{\eps}^j v^j_p + v^i_e \sum_{j = 1}^{i-1} \sqrt{\eps}^{j-1} u^j_{py} + u^i_{ex} \sum_{j = 0}^{i-1} \sqrt{\eps}^j \{u^j_p \\
& \hspace{15 mm} - u^j_p(x,\infty) \} + u^i_e \sum_{j = 0}^{i-1} \sqrt{\eps}^j u^j_{px} + \int_y^\infty \p_x \{ \sqrt{\eps}^2 u^i_e \sum_{j = 0}^{i-1} \sqrt{\eps}^j v^j_{px} + \sqrt{\eps} v^i_{ex} \\
& \hspace{15 mm} \times \sum_{j = 0}^{i-1} \sqrt{\eps}^j \{u^j_p - u^j_p(x,\infty)\} + \sqrt{\eps}^2 v^i_{eY} \sum_{j = 0}^{i-1} \sqrt{\eps}^j v^j_p + \sqrt{\eps} v^i_e \sum_{j = 0}^{i-1} \sqrt{\eps}^j v^j_{py} \\
& \hspace{15 mm} + \sqrt{\eps} v^{i-1}_s v^{i-1}_{py} + \sqrt{\eps} v_{sy}^{i-1} v^{i-1}_p  + \sqrt{\eps} v^{E,i-1}_{sx} \{u^{i-1}_p - u^{i-1}_p(x,\infty)\} \\
& \hspace{15 mm} +  \sqrt{\eps} v^{P,i-1}_{sx} u^{i-1}_{p} + \sqrt{\eps} u_s^{i-1} v^{i-1}_{px} + \sqrt{\eps} \Delta_\eps v^{i-1}_p + \sqrt{\eps}^i \{ u^{i-1}_p v^{i-1}_{px} + v^{i-1}_p v^{i-1}_{py} \} \} \ud z \\
& \hspace{15 mm} - g^{u, i}_{ext, p} + \int_y^\infty \p_x \{ \sqrt{\eps}^2 g^{v, i}_{ext, p} \} \ud z. 
\end{align*}
\end{definition}

For $i = 1$ only, we make the following modifications. The aim is to retain only the required order $\sqrt{\eps}$ terms into $f^{(1)}$. $f^{(2)}$ will then be adjusted by including the superfluous terms. Moreover, $f^{(1)}$ will contain the important $g^{u, 1}_{ext, p}$ external forcing term. Specifically, define: 
\begin{align} 
\begin{aligned} \label{defn.f1.special}
f^{(1)} := &g^{u, 1}_{ext, p} - u^0_p u^1_{ex}|_{Y = 0} - u^0_{px} u^1_e|_{Y = 0} \\
& - \bar{u}^0_{eY}(0) y u^0_{px} - v^0_p u^0_{eY} - v^1_{eY}(0) y u^0_{py}.
\end{aligned}
\end{align} 

\subsection{$i = n$}

For the final Prandtl layer, we must enforce the boundary condition $v^n_p|_{y = 0} = 0$. Define the quantities $[u_p, v_p, P_p]$ to solve
\begin{align} \label{des.pr.1}
\left.
\begin{aligned}
&\bar{u} \p_x u_p + u_p \p_x \bar{u} + \p_y \bar{u} v_p + \bar{v} \p_y u_p + \p_x P_p - \p_{yy} u_p := f^{(n)}, \\  
& \p_x u_p + \p_y v_p = 0,  \hspace{5 mm} \p_y P^i_p = 0\\  
& [u_p, v_p]|_{y = 0} = [-u^n_e, 0]|_{y = 0}, \hspace{5 mm} u_p|_{y \rightarrow \infty} = 0 \hspace{5 mm} v_p|_{x = 0} = V_P^n. 
\end{aligned}
\right\}
\end{align}

Note the change in boundary condition of $v_{p}|_{y = 0} = 0$ which contrasts the $i = 1,..,n-1$ case. This implies that $v_p = \int_0^y u_{px} \ud y'$. For this reason, we must cut-off the Prandtl layers: 
\begin{align*}
&u^n_p := \chi(\sqrt{\eps}y) u_p + \sqrt{\eps} \chi'(\sqrt{\eps}y) \int_0^y u_p(x, y') \ud y', \\
&v^n_p := \chi(\sqrt{\eps}y) v_p. 
\end{align*}

Here $\mathcal{E}^n$ is the error contributed by the cut-off: 
\begin{align*}
\mathcal{E}^{(n)} &:= \bar{u} \p_x u^{n}_{p} + u^n_p \p_x \bar{u}  +\bar{v} \p_y u^n_{p} + v^n_p \p_y \bar{u}  - u^n_{pyy} - f^{(n)}. 
\end{align*}

Computing explicitly: 
\begin{align} \n
\mathcal{E}^{(n)} := &(1-\chi) f^{(n)} + \bar{u} \sqrt{\eps} \chi'(\sqrt{\eps}y) v_p(x,y) + \bar{u}_{x} \sqrt{\eps} \chi' \int_0^y u_p \\  \n
& + \bar{v} \sqrt{\eps} \chi' u_p + \eps \bar{v} \chi'' \int_0^y u_p + \sqrt{\eps} \chi' u_p \\ \label{dan.1}
& + \eps^{\frac{3}{2}} \chi''' \int_0^y u_p + 2\eps \chi'' u_p + \sqrt{\eps} \chi' u_{py}.
\end{align}

We will now define the contributions into the next order, which will serve as the forcing for the remainder term: 
\begin{align} \n
&\underbar{f}^{(n+1)} := \sqrt{\eps}^n \Big[ \eps u^n_{pxx} + v^n_p\{ \bar{u}^n_{sy} - u^0_{py} \} + \{u^0_e - u^0_e(0) \} u^n_{px} \\ \n
& \hspace{15 mm} + u^n_{px} \sum_{j = 1}^n  \sqrt{\eps}^j (u^j_e + u^j_p) + \{ u^n_{sx} - \bar{u}^0_{sx} \} u^n_p + (v^n_s - v^1_s) u^n_{py} \\ \n
& \hspace{15 mm} + \{ v^1_e - v^1_e(x,0) \} u^n_{py} \Big] + \sqrt{\eps}^n \mathcal{E}^{(n)} + \sqrt{\eps}^{n+2} \Delta u^n_e \\ \label{underbar.f}
& \hspace{15 mm} + \sqrt{\eps}^n u^n_{ex} \sum_{j = 1}^{n-1} \sqrt{\eps}^j u^j_e + \sqrt{\eps}^n u^n_e \sum_{j = 1}^{n-1} \sqrt{\eps}^j u^j_{ex} + \sqrt{\eps}^{2n} [ u^n_e u^n_{ex} \\ \n
& \hspace{15 mm} + v^n_e u^n_{eY}] + \sqrt{\eps}^{n+1} u^n_{eY} \sum_{j= 1}^{n-1} \sqrt{\eps}^{j-1} v^j_e + \sqrt{\eps}^{n-1}v^n_e \sum_{j = 1}^{n-1} \sqrt{\eps}^{j+1} u^j_{eY} ..
\end{align}

\begin{align} \n
& \underbar{g}^{(n+1)} := \sqrt{\eps}^n \Big[ v_s^n \p_y v^n_p + \p_y v_s^n v^n_p + \p_x v^n_s u^n_p + u^n_s \p_x v^n_p - \Delta_\eps v^n_p \\ \n
& \hspace{15 mm}  + \sqrt{\eps}^n \Big( u^n_p \p_x v^n_p + v^n_p \p_y v^n_p \Big) \Big] + (\sqrt{\eps})^n \p_Y v^n_e \sum_{j = 1}^{n-1} (\sqrt{\eps})^{j-1} v^j_e \\ \label{underbar.g}
& \hspace{15 mm}+ \sqrt{\eps}^{n-1} v^n_e \sum_{j = 1}^{i-1} \sqrt{\eps}^j \p_Y v^j_e + \sqrt{\eps}^{2n-1} [v^n_e v^n_{eY} + u^n_e \p_x v^n_e]  \\ \n
& \hspace{15 mm} + \sqrt{\eps}^n u^n_e \sum_{j =1}^{n-1} (\sqrt{\eps})^{j-1}\p_x v^j_e + \sqrt{\eps}^{n-1} \p_x v^n_e \sum_{j = 0}^{n-1} \sqrt{\eps}^j u^j_e + \sqrt{\eps}^{n+1} \Delta v^n_e.
\end{align}

\subsection{Remainder System} \label{subsection.rem}

A straightforward linearization yields:
\begin{align} \label{rem.sys.1}
\left.
\begin{aligned} 
&-\Delta_\eps u^{(\eps)} + S_u + \p_x P^{(\eps)} = \eps^{-N_0} \underbar{f}^{(n+1)} + \eps^{N_0} \Big[ u^{(\eps)}\p_x u^{(\eps)} + v^{(\eps)} \p_y u^{(\eps)} \Big] \\
&-\Delta_\eps v^{(\eps)} + S_v + \frac{\p_y}{\eps}P^{(\eps)} =\eps^{-N_0} \underbar{g}^{(n+1)} + \eps^{N_0} \Big[ u^{(\eps)} \p_x v^{(\eps)} + v^{(\eps)} \p_y v^{(\eps)} \Big] \\
&\p_x u^{(\eps)} + \p_y v^{(\eps)} = 0.
\end{aligned}
\right\}
\end{align}

Denote: 
\begin{align}
u_s := \tilde{u}^n_s, \hspace{5 mm} v_s := \tilde{v}^n_s.
\end{align}

Here we have defined: 
\begin{align} \label{Su}
&S_u = u_s \p_x u^{(\eps)} + u_{sx}u^{(\eps)} + v_s \p_y u^{(\eps)} + u_{sy}v^{(\eps)}, \\ \label{Sv}
&S_v = u_s \p_x v^{(\eps)} + v_{sx}u^{(\eps)} + v_s \p_y v^{(\eps)} + v_{sy}v^{(\eps)}.
\end{align}

We will also define: 
\begin{align} \label{forcingdefn}
&\tilde{g}_{(u)} := \eps^{-N_0} \{ \p_{y} \underbar{f}^{(n+1)} - \eps \p_{x} \underbar{g}^{(n+1)} \}, \text{ and }\tilde{g}_{(q)} := \p_x \tilde{g}_{(u)}.
\end{align}

Going to the vorticity formulation of (\ref{rem.sys.1}), we obtain the following: 
\begin{align}  
\begin{aligned} \label{eqn.vort.A}
-R[q^{(\eps)}] - u^{(\eps)}_{yyy} &+ 2\eps v^{(\eps)}_{xyy} + \eps^2 v^{(\eps)}_{xxx} + A_1 + A_2 \\
& = \tilde{g}_{(u)} + \eps^{N_0} \{ v \Delta_\eps u - u \Delta_\eps v \},
\end{aligned}
\end{align}

where we have defined the Rayleigh operator:
\begin{align}
R[q^{(\eps)}] = \p_y\{ u_s^2 \p_y q^{(\eps)}\} + \eps \p_x \{ u_s^2 q^{(\eps)}_x \},
\end{align}

and where:
\begin{align}
A_1 := v_s u^{(\eps)}_{yy} - v_{syy}u^{(\eps)}, \hspace{3 mm} A_2 := \eps v_s u^{(\eps)}_{xx} - \eps v_{sxx}u^{(\eps)}. 
\end{align}

Note for future reference that we may alternatively write: 
\begin{align*}
A_1 + A_2 = v_s \Delta_\eps u^{(\eps)} - u^{(\eps)} \Delta_\eps v_s.
\end{align*}

In Section \ref{Section.1}, our main object of analysis with the vorticity equation evaluated at the $\{x = 0\}$ boundary, $(\ref{eqn.vort.A})|_{x = 0}$, which reads: 
\begin{align}
\begin{aligned} \label{sys.u0.app.unh}
&\mathcal{L} u^{0,\eps} := - u^{0,\eps}_{yyy} + v_s u^{0,\eps}_{yy} - u^{0,\eps} \Delta_\eps v_s = F, \\
&F := -2\eps u_s u_{sx}q^{(\eps)}_x|_{x = 0} - 2\eps v^{(\eps)}_{xyy}|_{x = 0} - \eps^2 v^{(\eps)}_{xxx}|_{x = 0} - \eps v_s u^{(\eps)}_{xx}|_{x = 0} + \tilde{g}_{(u)} , \\
& u^{0,\eps}(0) = 0, \p_y u^{0,\eps}(\infty) = 0, \p_{yy} u^{0,\eps}(\infty) = 0.
\end{aligned}
\end{align}

The $x$- differentiated vorticity equation, which we refer to as DNS, $\p_x (\ref{eqn.vort.A})$, reads:
\begin{align}
\begin{aligned}  \label{eqn.dif.1.app.unh}
&-\p_x R[q^{(\eps)}] + \Delta_\eps^2 v^{(\eps)} + \p_x \{A_i \} =\eps^{N_0} \mathcal{N} + \tilde{g}_{(q)}, \\
&v^{(\eps)}|_{x = 0} = a^{(\eps)}_0,  v^{(\eps)}_{xx}|_{x = 0} = a^{(\eps)}_2, v^{(\eps)}_x|_{x = L} = a^{(\eps)}_1, v^{(\eps)}_{xxx}|_{x = L} = a^{(\eps)}_3, \\
&v^{(\eps)}|_{y = 0} = v^{(\eps)}_y|_{y = 0} = 0. 
\end{aligned}
\end{align}

\noindent It is useful to consider $\mathcal{N} = \mathcal{N}(\bar{u}^0, \bar{v})$ which is more suitable to apply a contraction mapping argument. This therefore has the expression: 
\begin{align} 
\begin{aligned} \label{defn.nonlinear}
\eps^{N_0}\mathcal{N} = & \eps^{N_0} \Big( \bar{v}_y \Delta_\eps \bar{v} + I_x[\bar{v}_y] \Delta_\eps \bar{v}_x - \bar{v}_x I_x[\bar{v}_{yyy}] \\
&- \eps \bar{v}_x \bar{v}_{xy} - \bar{v} \Delta_\eps \bar{v}_y - \bar{u}^0 \Delta_\eps \bar{v}_x + \bar{v}_x \bar{u}^0_{yy}\Big).
\end{aligned}
\end{align}

The forcing $g_{(q)}$ has been defined above in (\ref{forcingdefn}). The term $b_{(q)}$ arises as a result of homogenizing the boundary conditions. Define the boundary corrector 
\begin{align*}
\tilde{v} := a^\eps_0 + x\{a^\eps_1 - L a^\eps_2 - \frac{L^2}{2}a^\eps_3  \} + x^2 \frac{a^\eps_2}{2} + x^3 \frac{a^\eps_3}{6}.
\end{align*}

\noindent   One checks immediately that $\tilde{v}$ achieves the boundary conditions in (\ref{eqn.dif.1.app.unh}). We homogenize at the level of the vorticity equation, (\ref{eqn.vort.A}). Define the homogenized quantities:
\begin{align} \label{defn.h}
v := v^{(\eps)} - \tilde{v}, \hspace{3 mm} u := u^{(\eps)} + \int_0^x \tilde{v}_y, \hspace{3 mm} u_s q := v.
\end{align}

\noindent It is clear that the divergence free condition is satisfied by the pair $[u, v]$. Writing (\ref{eqn.vort.A}) gives: 
\begin{align} 
\begin{aligned} \label{h.vort}
- R[q] - u_{yyy} &+ 2\eps v_{xyy} + \eps^2 v_{xxx} + \{v_s + \eps^{N_0} \tilde{v}\} \Delta_\eps u \\ 
& - u \Delta_\eps \{v_s + \eps^{N_0} \tilde{v} \} - \eps^{N_0} I_x[\tilde{v}_y] \Delta_\eps v + \eps^{N_0} v \Delta_\eps I_x[\tilde{v}_y] \\
& = g_{(u)} + b_{(u)} + \eps^{N_0} \{ \p_y N_u(u, v)  - \eps \p_x N_v(u, v) \}, 
\end{aligned}
\end{align}

\noindent where:
\begin{align*}
&b_{(u)} = - R[\tilde{v}] + I_x[\tilde{v}_{yyyy}] + 2\eps \tilde{v}_{xyy} + \eps^2 \tilde{v}_{xxx} - \eps \tilde{v}_{xy} + v_s I_x[\tilde{v}_{yyy}] - \Delta_\eps v_s I_x[\tilde{v}_y] , \\
&b_{(q)} = \p_x b_{(u)}.
\end{align*}

\noindent We then solve for $[u, v]$ using the equations (\ref{eqn.dif.1.app}) and (\ref{sys.u0.app}). By definition, the pair $[u := u^{0} - \int_0^x v_y, v]$ solves the vorticity equation, (\ref{h.vort}). We then obtain $[u^{(\eps)}, v^{(\eps)}]$ using (\ref{defn.h}). We now summarize the two systems we will be studying. First the $\mathcal{L}$-system:
\begin{align}
\begin{aligned} \label{sys.u0.app}
&\mathcal{L} u^0 := - u^{0}_{yyy} + v_s u^{0}_{yy} - u^{0} \Delta_\eps v_s = F, \\
&F := \underbrace{-2\eps u_s u_{sx}q_x|_{x = 0} - 2\eps v_{xyy}|_{x = 0} - \eps^2 v_{xxx}|_{x = 0} - \eps v_s u_{xx}|_{x = 0}}_{F_{(v)}} + g_{(u)}, \\
& u^{0,\eps}(0) = 0, \p_y u^{0,\eps}(\infty) = 0, \p_{yy} u^{0,\eps}(\infty) = 0.
\end{aligned}
\end{align}

\noindent and 
\begin{align}
\begin{aligned}  \label{eqn.dif.1.app}
&-\p_x R[q] + \Delta_\eps^2 v+ \p_x \{A_i \} = \p_x N + g_{(q)}, \\
&v|_{x = 0} = 0,  v_{xx}|_{x = 0} = 0, v_x|_{x = L} = 0, v_{xxx}|_{x = L} = 0, \\
&v|_{y = 0} = v_y|_{y = 0} = 0. 
\end{aligned}
\end{align}

Here, we take as forcing: 
\begin{align}
&g_{(u)} := \tilde{g}_{(u)} + b_{(u)} + H_{(u)}, \hspace{5 mm} g_{(q)} := \tilde{g}_{(q)} + b_{(q)} + H_{(q)}.
\end{align}

Finally, we define the quantities $H_{(u)}$ and $H_{(q)}$ that appear above: 
\begin{align}
\begin{aligned} \label{defn.H}
&H_{(u)} := u \Delta_\eps \eps^{N_0} \tilde{v} - \eps^{N_0} \tilde{v} \Delta_\eps u + \eps^{N_0} I_x[\tilde{v}_y] \Delta_\eps v - \eps^{N_0} v \Delta_\eps I_x[\tilde{v}_y], \\
&H_{(q)} = \p_x H_{(u)}.
\end{aligned}
\end{align}

Recall the definition of $\mathcal{F}$ from (\ref{defn.forcing.f}). The following estimates are clear that according to the assumption, (\ref{assume.bq.intro}): 
\begin{align} 
\begin{aligned}\label{bqbu}
&\mathcal{F}(b_{(q)}, b_{(u)}) \le o(1), \\
&\mathcal{F}(H_{(q)}, H_{(u)}) \le o(1) \| \bold{u} \|_{\mathcal{X}}^2.
\end{aligned}
\end{align}

The following proposition summarizes the profile constructions from \cite{}: 
\begin{theorem} \label{thm.construct} Assume the shear flow $u^0_e(Y) \in C^\infty$, whose derivatives decay rapidly. Assume (\ref{OL.1}) regarding $\bar{u}^0_p|_{x = 0}$, and the conditions
\begin{align}  \label{compatibility.1.fin}
& \bar{v}^i_{pyyy}|_{x = 0}(0) = \p_x g_1|_{x = 0, y = 0}, \\ \label{compatibility.2.fin}
&\bar{v}^i_p|_{x = 0}''''(0) = \p_{xy}g_1|_{y =0}(x = 0), \\ \label{integral.cond}
&\bar{u}^0_{p y}|_{x = 0}(0) u^{i}_e|_{x = 0}(0) - \int_0^\infty \bar{u}^0_{p} e^{-\int_1^y \bar{v}^0_{p}} \{f^{(i)}(y) - r^{(i)}(y) \} \ud y  = 0,
\end{align}

\noindent where $r^{(i)}(y) := \bar{v}^i_p \bar{u}^0_{p y} - \bar{u}^0_{p} \bar{v}^i_{py}$. We assume also standard higher order versions of the parabolic compatibility conditions (\ref{compatibility.1.fin}), (\ref{compatibility.2.fin}). Let $v^i_e|_{x = 0}, v^i_e|_{x = L}, u^i_e|_{x = 0}$  be prescribed smooth and rapidly decaying Euler data. We assume on the data standard elliptic  compatibility conditions at the corners $(0,0)$ and $(L,0)$ obtained by evaluating the equation at the corners. In addition, assume 
\begin{align}
&v^1_e|_{x = 0} \sim Y^{-m_1} \text{ or } e^{-m_1 Y} \text{ for some } 0 < m_1 < \infty,\\
& \|\p_Y^k \{ v^i_e|_{x = 0} - v^i_e|_{x = L} \} \langle Y \rangle^M\|_\infty \lesssim L
\end{align}

\noindent Then all profiles in $[u_s, v_s]$ exist and are smooth on $\Omega$. The following estimates hold: 
\begin{align}
\begin{aligned} \label{prof.pick}
&\bar{u}^0_p > 0,\bar{u}^0_{py}|_{y = 0} > 0, \bar{u}^0_{p yy}|_{y = 0} = \bar{u}^0_{p yyy}|_{y = 0} = 0 \\
&\| \nabla^K \{ u^0_p, v^0_p\} e^{My} \|_\infty \lesssim 1 \text{ for any } K \ge 0, \\
&\|u^i_p \|_\infty + \| \nabla^K  u^i_p  e^{My} \|_\infty + \| \nabla^J v^i_p e^{My} \|_\infty \lesssim 1 \text{ for any } K \ge 1, M \ge 0, \\
&\| \nabla^K \{u^1_e, v^1_e\}  w_{m_1} \|_\infty \lesssim 1 \text{ for some fixed } m_1 > 1  \\
&\| \nabla^K \{u^i_e, v^i_e\} w_{m_i}\|_\infty \lesssim 1 \text{ for some fixed } m_i > 1,
\end{aligned}
\end{align}

\noindent where $w_{m_i} \sim e^{m_i Y}$ or $(1+Y)^{m_i}$. 

In addition the following estimate on the remainder forcing (recall (\ref{defn.forcing.f})) holds: 
\begin{align} \label{thm.force.maz}
\mathcal{F}  \lesssim \sqrt{\eps}^{n-1-2N_0}.
\end{align}
\end{theorem}

\noindent \textbf{Acknowledgements:} This research is supported in part by NSF grants DMS-1611695, DMS-1810868, Chinese
NSF grant 10828103, BICMR, as well as a Simon Fellowship. Sameer Iyer was also supported in part by NSF grant DMS 1802940.

\def\bibindent{3.5em}


\begin{thebibliography}{20\kern\bibindent}

\setlength{\parskip}{0pt}
\setlength{\itemsep}{0pt}

\bibitem[Ad03]{Adams} R. Adams, J. Fournier, \textit{Sobolev Spaces.} Second Edition, Elsevier Science Ltd, Kidlington, Oxford, (2003).

\bibitem[AWXY15]{AL}
R. Alexandre, Y.-G. Wang, C.-J. Xu, T. Yang, Well-posedness of the Prandtl equation in Sobolev spaces. \textit{Journal of the American Mathematical Society, } Volume 28, Number 3, (2015), P. 745-784.

\bibitem[As91]{Asano}
A. Asano, Zero viscosity limit of incompressible Navier-Stokes equations, \textit{Conference at the Fourth Workshop on Mathematical Aspects of Fluid and Plasma Dynamics,} Kyoto, 1991. 

\bibitem[BT13]{BardosTiti}
C. Bardos and E. Titi. Mathematics and turbulence: where do we stand? \textit{J. Turbul.,} 14(3):42?76, (2013). 22, 23.

\bibitem[BR80]{Biharmonic}
H. Blum, R. Rannacher, On the Boundary Value Problem of the Biharmonic Operator on Domains with Angular Corners, \textit{Math. Meth. in the Appl. Sci.}, (1980), 556-581.

\bibitem[DM15]{MD}
A.L. Dalibard, and N. Masmoudi ``Phenomene de separation pour l'equation de Prandtl stationnaire." Seminaire Laurent Schwartz - EDP et applications (2014-2015): 1-18.

\bibitem[DPV12]{Fractional}
E. Di Nezza, G. Palatucci, E. Valdinoci. Hitchhiker's guide to fractional Sobolev spaces, \textit{Bulletin des Sciences Mathematiques}, 136 (2012).


\bibitem[E00]{E} W. E, Boundary layer theory and the zero-viscosity limit of the Navier-Stokes equation. \textit{Acta Math. Sin.} (Engl. Ser.) 16, 2 (2000), P. 207-218.

\bibitem[EE97]{EE} W. E, and B. Engquist. Blowup of solutions of the unsteady Prandtl's equation. \textit{Comm. Pure Appl. Math.,} 50, 12 (1997), P. 1287-1293.

\bibitem[EV10]{Evans}
 L.C. Evans, \textit{Partial Differential Equations}. Second Edition. Graduate Studies in Mathematics. American Mathematical Society, Providence, RI, (2010).

\bibitem[LMN08]{HLop}
M.C. Lopes Filho, A.L. Mazzucato, H.J. Nussenzveig Lopes. Vanishing viscosity limit for incompressible flow inside a rotating circle. \textit{Physica D} 237 (2008) 1324 - 1333.

\bibitem[Galdi11]{Galdi}
G.P. Galdi, \textit{An Introduction to the Mathematical Theory of the Navier-Stokes Equations.} Second Edition. Springer Monographs in Mathematics. (2011).

\bibitem[GVD10]{GVD} 
D. Gerard-Varet and E. Dormy. On the ill-posedness of the Prandtl equation. \textit{Journal of the American Mathematical Society,} Volume 23, Number 2, (2010), P. 591-609.  

\bibitem[GVM13]{GVM} D. Gerard-Varet and N. Masmoudi. Well-posedness for the Prandtl system without analyticity or monotonicity. arXiv preprint 1305.0221v1, (2013). 

\bibitem[GVMM16]{DMM}
D. Gerard-Varet, Y. Maekawa, and N. Masmoudi. Gevrey stability of Prandtl expansions for 2D Navier-Stokes flows. arXiv preprint 1607.06434v1, (2016). 

\bibitem[GVM18]{Varet-Maekawa}
D. Gerard-Varet and Y. Maekawa. Sobolev stability of Prandtl expansions for the steady Navier-Stokes equations. arXiv preprint 1805.02928, (2018).

\bibitem[GVN12]{GVN}
D. Gerard-Varet and T. Nguyen, Remarks on the ill-posedness of the Prandtl equation. \textit{Asymptotic Analysis,} 77, no.1-2, P. 71-88, (2012). 

\bibitem[GJT16]{Temam}
G-M. Gie, C-Y. Jung, and R. Temam, Recent Progresses in the Boundary Layer Theory. \textit{Disc. Cont. Dyn. Syst. A,} 36, 5 (2016), P. 2521 - 2583. 

\bibitem[GGN15a]{GGN1} E. Grenier, Y. Guo, T. Nguyen, Spectral instability of symmetric shear flows in a two-dimensional channel, \textit{Advances in Mathematics}, Volume 292, 9 April 2016, P. 52-110. 

\bibitem[GGN15b]{GGN2} E. Grenier, Y. Guo, T. Nguyen, Spectral instability of characteristic boundary layer flows.  \textit{Duke Math. J}, 165 (2016), no. 16 , 3085-3146.  

\bibitem[GGN15c]{GGN3} E. Grenier, Y. Guo, T. Nguyen, Spectral instability of Prandtl boundary layers: an overview. \textit{Analysis (Berlin),} 35, no.4, (2015), P. 343-355.

\bibitem[GrNg17a]{GrNg1} E. Grenier, T. Nguyen, On nonlinear instability of Prandtl's boundary layers: the case of Rayleigh's stable shear flows. arXiv preprint 1706.01282 (2017).

\bibitem[GrNg17b]{GrNg2} E. Grenier, T. Nguyen, Sublayer of Prandtl boundary layers. arXiv preprint 1705.04672 (2017).

\bibitem[GrNg18]{GrNg3} E. Grenier, T. Nguyen, $L^\infty$ Instability of Prandtl layers. arXiv preprint 1803.11024 (2018).

\bibitem[GI18]{GI1} Y. Guo, S. Iyer, Validity of Steady Prandtl Layer Expansions. arXiv preprint 1805.05891 (2018).

\bibitem[GN14]{GN}
Y. Guo, and T. Nguyen. Prandtl boundary layer expansions of steady Navier-Stokes flows over a moving plate. \textit{Annals of PDE} (2017) 3: 10. doi:10.1007/s40818-016-0020-6.

\bibitem[GN11]{GN2}
Y. Guo and T. Nguyen. A note on the Prandtl boundary layers. \textit{Comm. Pure Appl. Math.,} 64 (2011), no. 10, P. 1416-1438.

\bibitem[HH03]{Hunter}
L. Hong and J. Hunter. Singularity formation and instability in the unsteady inviscid and viscous Prandtl equations. \textit{Comm. Math. Sci.} Vol. 1, No. 2, P. 293-316.

\bibitem[IV16]{Vicol}
M. Ignatova and V. Vicol. Almost Global Existence for the Prandtl Boundary Layer Equations. \textit{Arch. Rational Mech. Anal.,} 220, (2016) P. 809-848. 

\bibitem[Iy15]{Iyer}
S. Iyer. Steady Prandtl Boundary Layer Expansion of Navier-Stokes Flows over a Rotating Disk. \textit{Arch. Ration. Mech. Anal.}, 224 (2017), no. 2, 421-469.  

\bibitem[Iy16]{Iyer2}
S. Iyer. Global Steady Prandtl Expansion over a Moving Boundary. arXiv preprint 1609.05397, (2016). 

\bibitem[Iy17]{Iyer3}
S. Iyer. Steady Prandtl Layers over a Moving Boundary: Non-Shear Euler Flows. arXiv preprint 1705.05936v2, (2017).

\bibitem[Kel08]{Kelliher}
J. Kelliher. On the vanishing viscosity limit in a disk. \textit{Math. Ann.} (2009) 343:701-726. 

\bibitem[KMVW14]{KMVW}
I. Kukavica, N. Masmoudi, V. Vicol, T.K. Wong, On the local well-posedness of the Prandtl and Hydrostatic Euler equations with multiple monotonicity regions. \textit{SIAM Journal of Mathematical Analysis.,} Vol. 46, No. 6, P. 3865-3890.

\bibitem[KV13]{Kuka}
I. Kukavica and V. Vicol, On the local existence of analytic solutions to the Prandtl boundary layer equations. \textit{Commun. Math. Sci.,} 11(1), (2013), P. 269-292. 

\bibitem[KVW15]{KVW}
I. Kukavica, V. Vicol, F. Wang, The van Dommelen and Shen singularity in the Prandtl equations. arXiv preprint 1512.07358v1, (2015).

\bibitem[LCS03]{Lom}
M. C. Lombardo, M. Cannone and M. Sammartino, Well-posedness of the boundary
layer equations, \textit{SIAM J. Math. Anal.} 35(4), P. 987-1004, 2003.

\bibitem[Mae14]{Mae}
Y. Maekawa, On the inviscid problem of the vorticity equations for viscous incompressible flows in the half-plane, \textit{Comm. Pure. Appl. Math., 67,} (2014), 1045-1128.  

\bibitem[MM17]{MM}
Y. Maekawa and A. Mazzucato, The inviscid limit and boundary layers for Navier-Stokes flows, \textit{Handbook of Mathematical Analysis in Mechanics of Viscous Fluids} (2017), 1-48.

\bibitem[MT08]{Taylor}
A. Mazzucato and M. Taylor. Vanishing viscocity plane parallel channel flow and related singular perturbation problems. \textit{Anal. PDE, 1 (1); 35-93} (2008).

\bibitem[MW15]{MW}
N. Masmoudi and T.K. Wong, Local in time existence and uniqueness of solutions to the Prandtl equations by energy methods, \textit{Comm. Pure. Appl. Math., 68,} (2015), 1683-1741.

\bibitem[NI59]{GNS1}
L. Nirenberg. On elliptic partial differential equations. \textit{Annali Della Scoula Normale Superiore Di Pisa.,} 1959. 

\bibitem[OS99]{Oleinik}
Oleinik, O.A. and Samokhin , V.N, \textit{Mathematical models in boundary layer theory.} Applied Mathematics and Mathematical Computation, 15. Champan and Hall/ CRC, Boca Raton, FL, 1999. 

\bibitem[Ol67]{Oleinik1}
O.A. Oleinik. On the mathematical theory of boundary layer for an unsteady flow of incompressible fluid. \textit{J. Appl. Math. Mech.,} 30, (1967), P. 951-974.

\bibitem[Orlt98]{Orlt1}
M. Orlt, Regularity for Navier-Stokes in domains with corners, Ph.D. Thesis, 1998 (in German). 

\bibitem[OS93]{OS} 
M. Orlt, and A.M. Sandig, Regularity of viscous Navier-Stokes flows in non smooth domains. Boundary value problems and integral equations in non smooth domains (Luminy, 1993). \textit{Lecture Notes in Pure and Appl. Math.}, 167, Dekker, New York, 1995. 

\bibitem[Pr1905]{Prandtl}
L. Prandtl, Uber flussigkeits-bewegung bei sehr kleiner reibung. In \textit{Verhandlungen des III Internationalen Mathematiker-Kongresses, Heidelberg}. Teubner, Leipzig, 1904, pp. 484 - 491. 

English Translation: ``Motion of fluids with very little viscosity," Technical Memorandum No. 452 by National Advisory Committee for Aeuronautics. 

\bibitem[SC98]{Caflisch1} 
M. Sammartino and R. Caflisch, Zero viscosity limit for analytic solutions of the Navier-Stokes equation on a half-space. I. Existence for Euler and Prandtl equations. \textit{Comm. Math. Phys.} 192 (1998), no. 2, 433-461.

\bibitem[SC98]{Caflisch2}
M. Sammartino and R. Caflisch, Zero viscosity limit for analytic solutions of the Navier-Stokes equation on a half-space. II. Construction of the Navier-Stokes soution. \textit{Comm. Math. Phys. }192 (1998), no. 2, 463 - 491.

\bibitem[Sch00]{Schlicting}
Schlichting, H and Gersten, K. \textit{Boundary Layer Theory.} 8th edition, Springer-Verlag 2000. 

\bibitem[Ser66]{Serrin}
J. Serrin, Asymptotic behaviour of velocity profiles in the Prandtl boundary layer theory. \textit{Proc. R. Soc. Lond. A}, 299, 491-507, (1967).

\bibitem[TW02]{TWang}
R. Temam and T. Wang, Boundary layers associated with Incompressible Navier-Stokes equations: the noncharacteristic boundary case. \textit{J. Differ. Equations.}, 179, (2002), 647-686.

\bibitem[XZ04]{Xin}
Z. Xin and L. Zhang, On the global existence of solutions to the Prandtl's system. \textit{Adv. Math.,} Volume 181, Issue 1, (2004), P. 88-133. 

\end{thebibliography}
\end{document}